\documentclass[a4paper,12pt,reqno,oneside]{amsart}
\usepackage{setspace} 
\usepackage{typearea} 

\usepackage{amscd,amsmath,amssymb,verbatim}
\usepackage{xr-hyper}
\usepackage{graphicx}
\usepackage{ifthen}
\usepackage{paralist,cite}
\usepackage[colorlinks,backref,final,bookmarksnumbered,bookmarks]{hyperref}
\usepackage[backrefs]{amsrefs}
\usepackage{nicematrix}
\usepackage{tikz-cd}
\usepackage{accents}
\usepackage{enumitem}

\usepackage[utf8]{inputenc}
\usepackage[T1,T2A]{fontenc}

\newcommand{\arxiv}[2][]{\ifthenelse{\equal{#1}{}}
{\href{http://arxiv.org/abs/#2}{\tt arXiv:#2}}
{\href{http://arxiv.org/abs/math/#2}{\tt arXiv:math.#1/#2}}}

\theoremstyle{plain}
\newtheorem{maintheorem}{Theorem}
 
\newtheorem{theorem}{Theorem}[section]
\newtheorem{lemma}[theorem]{Lemma}

\newtheorem*{corollary*}{Corollary}
\newtheorem{proposition}[theorem]{Proposition}
\newtheorem{problem}[theorem]{Problem}

\theoremstyle{definition}

\newtheorem{example}[theorem]{Example}

\newtheoremstyle{remark}
{}{}{}{}{\itshape}{}{ }{\thmname{#1}\thmnumber{ \itshape #2.}}
\theoremstyle{remark}
\newtheorem{remark}[theorem]{Remark}

\def\N{\mathbb{N}} 
\def\R{\mathbb{R}} 
\def\Z{\mathbb{Z}}

\def\K{\mathcal{K}}

\def\QQ{\mathcal{Q}}
\def\PP{\mathcal{P}}

\def\Ho{\mathsf{Ho}}

\def\x{\times}
\def\but{\setminus}

\def\phi{\varphi} 

\def\xr#1{\xrightarrow{#1}} 
 \renewcommand{\:}{\colon}

\DeclareMathOperator*{\colim}{colim}

\DeclareMathOperator{\im}{im} 
\DeclareMathOperator{\Int}{Int} \DeclareMathOperator{\id}{id}

\def\tph#1{\raise2.5pt\hbox{\the\textfont1\char"7F}\!\!#1}
\def\tpm#1{\raise0pt\hbox{\the\textfont1\char"7F}\!#1}
\def\tpl#1{\lower1.5pt\hbox{\the\textfont1\char"7F}\!#1}

\font\ssym=cmbsy5 at 2.5pt 

\def\striangle{\text{\ssym\char"34}}

\def\sing#1#2{\smash{\overset\striangle{\smash{#1}
\vphantom{\vrule height#2pt}}}}

\def\spi{\sing\pi{3.5}}

\def\bydef{\mathrel{\mathop:}=}

\DeclareSymbolFont{bskadd}{U}{bskma}{m}{n}
\DeclareFontFamily{U}{bskma}{\skewchar\font130 }
\DeclareFontShape{U}{bskma}{m}{n}{<->bskma10}{}
\DeclareMathSymbol{\varlrttriangle}{\mathord}{bskadd}    {"E4}

\makeatletter
\newcommand*\nullseq{\mathop{\mathpalette\@biguoperator{\bigsqcup}}}
\newcommand*\@biguoperator[2]{\ooalign{\hidewidth$#1$\raisebox{1pt}{\scriptsize$\star$}\hidewidth\cr$#1#2$\cr}}
\makeatother
\begin{document}

\title{Fine shape. II: A Whitehead-type theorem}
\author{Sergey A. Melikhov}
\address{Steklov Mathematical Institute of Russian Academy of Sciences,
ul.\ Gubkina 8, Moscow, 119991 Russia}
\email{melikhov@mi-ras.ru}

\begin{abstract} We prove an ``abelian, locally compact'' Whitehead theorem in fine shape: 
A fine shape morphism between locally connected finite-dimensional locally compact separable 
metrizable spaces with trivial $\pi_0$ and $\pi_1$ is a fine shape equivalence 
if and only if it induces isomorphisms on the $\pi_i$ (=the Steenrod--Sitnikov homotopy groups).

We show by an example that the hypothesis of local connectedness cannot be dropped (even though
it can be dropped in the compact case).

As a byproduct, we also show that for a locally compact separable metrizable space $X$,
the Steenrod--Sitnikov homology $H_n(X)=0$ if and only if each compactum $K\subset X$ lies 
in a compactum $L\subset X$ such that the map $H_n(K)\to H_n(L)$ is trivial.

A cornerstone result of the paper is purely algebraic: 
If a direct sequence of groups $\Gamma_0\to\Gamma_1\to\dots$ has trivial colimit, then it is trivial 
as an ind-group (i.e.\ each $\Gamma_i$ maps trivially to some $\Gamma_j$), as long as it has one of 
the following forms:

$\bullet$ $\lim^1_i G_{i0}\to\lim^1_i G_{i1}\to\dots$, where the $G_{ij}$ are countable abelian groups;

$\bullet$ $\lim_i G_{i0}\to\lim_i G_{i1}\to\dots$, where the $G_{ij}$ are finitely generated groups,
which are either all abelian or satisfy the Mittag-Leffler condition for each $j$.
\end{abstract}

\maketitle

\section{Introduction}

\subsection{Algebra}

The following assertion is obvious:

\begin{proposition} \label{obvious}
If $\Gamma_0\to\Gamma_1\to\dots$ is a direct sequence of finite pointed sets with trivial $\colim\Gamma_i$, then 
for each $j$ there exists a $k>j$ such that the map $\Gamma_j\to\Gamma_k$ is trivial.
\end{proposition}

Not only finite pointed sets satisfy this property.

\begin{maintheorem} \label{bounded-colim}
Let us consider a direct sequence of towers of countable groups:
\[\begin{tikzcd}[row sep=1.5em,column sep=1.5em]
\vdots\dar&\vdots\dar&\vdots\dar& \\
G_{20}\rar\dar&G_{21}\rar\dar&G_{22}\rar\dar&\dots\\
G_{10}\rar\dar&G_{11}\rar\dar&G_{12}\rar\dar&\dots\\
G_{00}\rar&G_{01}\rar&G_{02}\rar&\dots.
\end{tikzcd}\]
Let us consider the following cases:

(a) Let $\Gamma_j=\lim_i G_{ij}$. 
Suppose that each $G_{ij}$ is abelian and finitely generated. 

(b) Let $\Gamma_j=\lim_i G_{ij}$. 
Suppose that $\lim^1_i G_{ij}$ vanish and each $G_{ij}$ is finitely generated.

(c) Let $\Gamma_j=\lim^1_i G_{ij}$. 
Suppose that each $G_{ij}$ is abelian.

(d) Let $\Gamma_j=\lim_i\lim^1_k G_{ij}^{(k)}$, where $G_{ij}^{(k)}=\im(G_{kj}\to G_{ij})$.

\noindent
In each of (a)--(d), if $\colim_j\Gamma_j$ is trivial, then for each $j$ there exists a $k>j$ such that 
the map $\Gamma_j\to\Gamma_k$ is trivial.
\end{maintheorem}

The hypothesis that each $G_{ij}$ is finitely generated cannot be dropped in (a) and (b).%
\footnote{If $G_{ij}=\bigoplus_{k=j}^\infty\Z$, each vertical bonding map
$G_{i+1,j}\to G_{ij}$ is an isomorphism, and each horizontal bonding map
$G_{i,j+1}\to G_{ij}$ is the projection, then $\colim_j\lim_i G_{ij}=0$ but
no map $\lim_i G_{ij}\to\lim_i G_{ik}$ is trivial.}

The hypothesis that each $G_{ij}$ is abelian also cannot be dropped in (a), as shown in 
Example \ref{unbounded-colim}.
The same example shows that the hypothesis that $\lim^1_i G_{ij}$ vanishes for each $j$ cannot be dropped in (b).

The hypothesis that the $G_{ij}$ are groups (rather than pointed sets) also cannot be dropped in (b), as one can see
from Example \ref{comb-example2}(b).%
\footnote{More precisely, let $G_{ij}$ be finite pointed sets forming the diagram of Theorem \ref{bounded-colim}. 
(Let us note that in this case each vertical tower $\dots\to G_{1j}\to G_{0j}$ satisfies the Mittag-Leffler condition.) 
Then it may happen that $\colim_j\Gamma_j$ is trivial, where $\Gamma_j=\lim_i G_{ij}$, but no map $\Gamma_j\to\Gamma_k$ is trivial.
Namely, this happens for $G_{ij}$ equal to the $\pi_0$ of appropriate nerves $P_{ij}$ of appropriate compact subsets $K_j$ 
of the locally compact space of Example \ref{comb-example2}(b).}

The proof of Theorem \ref{bounded-colim} is based on ``summation of an infinite series''.
Part (c) is deduced from (d), and the proof of (d) is the most difficult step.
We also provide a much shorter proof of the special case of (c) where all vertical maps in the diagram are injective
(Proposition \ref{bounded-colim2}).%
\footnote{Let us note that if all vertical maps in the diagram are surjective, then (c) holds trivially.}
While this special case is not directly used in the other proofs, the proof of (d) essentially combines 
a version of the proof of Proposition \ref{bounded-colim2} with an argument similar to the proof of (a).
At any rate the proofs of (a) and (b) involve working with two parameters, the proof of Proposition \ref{bounded-colim2}
with three, and the proof of (d) with four --- so these proofs should be easier to follow when read in this order.

The proof of Theorem \ref{bounded-colim}(c) additionally depends on a ``half'' of the following lemma.

\begin{maintheorem} \label{factorization-main}
Let $f=(f_i\:A_i\to C_i)_{i\in\N}$ be a level morphism between towers of abelian groups.
Suppose that $f$ factors through a tower of abelian groups $B_i$ such that the maps 
$\lim^1 A_i\to\lim^1 B_i$ and $\lim B_i\to\lim C_i$ are trivial.
Then $f$ also factors through a tower of abelian groups $G_i$ such that $\lim G_i=0$
and $\lim^1 G_i=0$.
\end{maintheorem}

The proof of Theorem \ref{factorization-main}, and particularly of its ``half'' which is used in the proof
of Theorem \ref{bounded-colim}(c), does not seem to generalize to the non-abelian case 
(see Remark \ref{level-factorization2}).
However, the following questions remain open.

\begin{problem}
Is Theorem \ref{factorization-main} true in the non-abelian case?
\end{problem}

\begin{problem}
Is Theorem \ref{bounded-colim}(c) true in the non-abelian case?
\end{problem}

\begin{problem}
Let $G_{ij}$ be finitely generated groups forming the diagram of Theorem \ref{bounded-colim}.
Suppose that $\colim_j\lim_i G_{ij}$ and $\colim_j\lim^1_i G_{ij}$ are trivial.
Then is it true that for each $j$ there exists a $k>j$ such that the maps $\lim_i G_{ij}\to\lim_i G_{ik}$ and 
$\lim^1_i G_{ij}\to\lim^1_i G_{ik}$ are trivial?
\end{problem}

\subsection{Steenrod homology and homotopy}

A {\it local compactum} is a locally compact separable metrizable space.
A {\it compactum} is a compact metrizable space.

For a compactum $K$ we denote by $H_n(K)$ its $n$th Steenrod homology group and by $\pi_n(K,x)$ its $n$th Steenrod homotopy set 
(group for $n>0$), which is the group/set of strong shape morphisms from $(S^n,x)$ to $(K,x)$.

For a metrizable space $X$ its $n$th Steenrod--Sitnikov homology group $H_n(X)$ is the direct limit $\colim_K H_n(K)$ 
over all compact $K\subset X$, and its $n$th Steenrod--Sitnikov homotopy group/set $\pi_n(X,x)$ is the direct limit $\colim_K\pi_n(K,x)$ 
over all compact $K\subset X$ containing $x$.

\begin{maintheorem} \label{ind-colimit}
Let $X$ be a local compactum and let $X_0\subset X_1\subset\dots$ be compact subsets of 
$X$ such that $\bigcup_i X_i=X$ and each $X_i\subset\Int X_{i+1}$.

(a) $H_n(X)=0$ if and only if for each $j$ there exists a $k\ge j$ such that the inclusion induced map 
$H_n(X_j)\to H_n(X_k)$ is trivial.

(b) Let $x\in X_0$.
Suppose that $X$ is locally connected and either $n\le 1$ or $\pi_1(X,x)=1$.
Then $\pi_n(X,x)$ is trivial if and only if for each $j$ there exists a $k\ge j$ such that the inclusion induced map 
$\pi_n(X_j,x)\to\pi_n(X_k,x)$ is trivial.
\end{maintheorem}

When $X$ is not locally connected, (b) fails for $n=0$ by Example \ref{comb-example2}(b) and for each $n\ge 1$ by
Example \ref{unbounded-colim2}.
When $\pi_1(X,x)\ne 1$, (b) fails for each $n\ge 2$, already when $X$ is a polyhedron, by using
the Stallings--Bieri groups (see Example \ref{stallings}).

\begin{maintheorem} \label{ind-isomorphism2}
Let $f\:X\to Y$ be a map between local compacta, let 
$X_0\subset X_1\subset\dots$ and $Y_0\subset Y_1\subset\dots$ be compact subsets of 
$X$ and $Y$ respectively such that $\bigcup_i X_i=X$ and $\bigcup_i Y_i=Y$ and each 
$X_i\subset\Int X_{i+1}$ and each $Y_i\subset\Int Y_{i+1}$ and also each $f(X_i)\subset Y_i$.
Let $f_i\:X_i\to Y_i$ be the restriction of $f$.

(a) $f_*\:H_n(X)\to H_n(Y)$ is an isomorphism for all $n$ if and only if the induced maps $f_{i*}\:H_n(X_i)\to H_n(Y_i)$ 
represent an ind-isomorphism for all $n$.

(b) Suppose that $X$ and $Y$ are locally connected and $\pi_1(X,x)=\pi_1\big(Y,f(x)\big)=1$ for each $x\in X$.
Then $f_*\:\pi_n(X,x)\to\pi_n\big(Y,f(x)\big)$ is an isomorphism for all $n$ and all $x\in X$ if and only if 
for each $x\in X$ the induced maps $f_{i*}\:\pi_n(X_i,x)\to\pi_n\big(Y_i,f(x)\big)$, where $i\ge\min\{j\mid x\in X_j\}$, 
represent an ind-isomorphism for all $n$.
\end{maintheorem}

Let us note that $n$ and $x$ are fixed in Theorem \ref{ind-colimit}, but vary in Theorem \ref{ind-isomorphism2}.

\begin{problem} Is Theorem \ref{ind-isomorphism2} true with $n$ fixed? With $x$ fixed? With both $n$ and $x$ fixed?
\end{problem} 

\begin{remark} The proofs of Theorem \ref{ind-colimit}(a) and Theorem \ref{ind-isomorphism2}(a) work for any
homology theory (in the sense of the Eilenberg--MacLane axioms without the dimension axiom) that satisfies
Milnor's map excision axiom on compacta \cite{Mi1} (see also \cite{M00}) and has compact supports
(i.e.\ the natural homomorphism $\colim_K H_n(K)\to H_n(X)$, where $K$ runs over all compact subspaces of $X$, is an isomorphism.)
\end{remark}

\subsection{Fine shape}

For the case of compacta, there is a fully satisfactory Whitehead theorem in fine shape (which coincides with strong shape in this case):

\begin{theorem} \cite{M1} \label{comp-Wh}
Let $X$, $Y$ be connected finite-dimensional compacta.
If $f\:(X,x)\to (Y,y)$ is a strong shape morphism inducing bijections on the $\pi_i$, then $f$ is 
a strong shape equivalence.
\end{theorem}

It should be noted that an erroneous proof of Theorem \ref{comp-Wh} appeared much earlier in \cite{Li} (whose author thought that $\pi_1(P,Q,x)$ 
is a group, considered ``$\lim^1$'' of such ``groups'', and took its triviality to imply the Mittag-Leffler condition).
The correct part of the arguments in \cite{Li} (see also \cite{Koy}) proves an ``abelian'' version of Theorem \ref{comp-Wh}.

While there are quite many ``Whitehead theorems in shape theory'' in the literature, most of them are in terms of pro-groups, 
rather than genuine groups, so of such results Theorem \ref{comp-Wh} appears to be by far the most advanced one.

All ``Whitehead theorems in shape theory'' that are known to the author are restricted to compacta, and it seems
highly unlikely that the literature contains any substantial result of this type beyond the compact case ---
because fine shape appeared only recently, and all alternative shape theories of non-compact spaces seem to be 
far less tractable.

In this paper we prove a ``non-compact Whitehead theorem in shape theory'', but with significant restrictions of 
two types: topological (local compactness) and algebraic (coming from the restrictions of Theorems \ref{bounded-colim}
and \ref{factorization-main}, which force us to deal only with finitely generated abelian groups).

\begin{maintheorem} \label{Wh-thm}
Let $X$ and $Y$ are locally connected finite-dimensional local compacta with trivial $\pi_0$ and $\pi_1$.
If $f\:(X,x)\to (Y,y)$ is a fine shape morphism inducing bijections on the $\pi_i$, then $f$ 
is a fine shape equivalence.
\end{maintheorem}

By Example \ref{comb-example2} the hypothesis that $X$ and $Y$ are locally connected cannot be dropped.

The hypothesis that $\pi_0$ and $\pi_1$ are trivial can be slightly relaxed (see Theorem \ref{Wh-thm1} and Lemma \ref{Cech}),
but the following remains a mystery to the author:

\begin{problem} Can the hypothesis ``$X$ and $Y$ have trivial $\pi_0$ and $\pi_1$'' be relaxed to ``$X$ and $Y$ are connected'' 
in Theorem \ref{Wh-thm}?
\end{problem}

For local compacta fine shape reduces to strong antishape \cite{M-I}*{Theorem \ref{fish:fish-sash}},
which can be said to belong to the domain of combinatorial homotopy.
So it appears that it should be possible to solve the following problem.

\begin{problem}
Find an algebraic characterization of fine shape equivalences among all fine shape morphisms
between two given finite-dimensional connected local compacta.
\end{problem}

\section{Examples}

\subsection{The comb space}

\begin{example} \label{comb-example2}
(a) Let $X\subset\R^2$ be the comb-and-flea set
$\{\frac1n\mid n\in\N\}\x[0,1]\cup (0,1]\x\{1\}\cup\{(0,0)\}$, and let $Y=X\cup E$,
where $E=\{-1\}\x[0,1]\cup\{0,1\}\x[-1,0]$.
It is not hard to see that $\tilde H_0(Y)\simeq\prod_{i\in\N}\Z\big/\bigoplus_{i\in\N}\Z$ 
(compare \cite{M-V}*{Example \ref{cor:comb1}}).
Hence $Y$ is not fine shape equivalent to a point.
On the other hand, writing $x=(0,0)$, it is easy to see that $\pi_n(Y,x)$ is trivial for all $n$ 
(compare \cite{M00}*{Remark \ref{book:ssh-ash2}}).

(b) Let $Z=Y\cup\{0\}\x[0,\frac12)$. 
Then $Z$ is locally compact and $(Z,x)$ is homotopy equivalent to $(Y,x)$.
Thus $\pi_n(Z,x)$ is trivial for all $n$, but $(Z,x)$ is not fine shape equivalent to a point.
\end{example}

\subsection{The polyhedral case}

\begin{example}\label{stallings}
For each $n\ge 2$ there exists a locally compact polyhedron $X$, which is the union of its compact subpolyhedra 
$X_0\subset X_1\subset\dots$ such that each $X_i\subset\Int X_{i+1}$, and a basepoint $x\in X_0$ such that 
$\pi_n(X,x)=0$ but the map $\pi_n(X_0,x)\to\pi_n(X_i,x)$ is not trivial for any $i$.

Indeed, let $G_n$ be the Stallings--Bieri group, that is, the kernel of the homomorphism $(F_2)^{n+1}\to\Z$
that sends all the generators to $1$, where $F_2$ be the free group of rank $2$.
It is well-known that there exists a countable CW complex $L$ of the type $K(G_n,1)$ with a finite $n$-skeleton, 
but there does not exists one with a finite $(n+1)$-skeleton (see \cite{Bri}).
Since $L$ has countably many cells, $L=\bigcup_{i=1}^\infty K_i$, where $K_0\subset K_1\subset\dots$ are its finite subcomplexes.
Let $x\in K_0$.
We may assume that $K_0$ contains the $n$-skeleton.
Then the map $\pi_n(K_0,x)\to\pi_n(K_i,x)$ is surjective for each $i$.
Suppose that it is trivial for some $i$.
Then $\pi_n(K_i,x)=0$.
Then we can obtain a $K(G_n,1)$ from $K_i$ by attaching cells of dimension $\ge n+2$.
Thus it has a finite $(n+1)$-skeleton, which is a contradiction.
Thus the map $\pi_n(K_0,x)\to\pi_n(K_i,x)$ is not trivial for any $i$.
Now let $X$ be the mapping telescope $K_{[0,\infty)}$, and let $X_i=K_{[0,i]}$.
Then $X$ is a local compactum and each $X_i\subset\Int X_{i+1}$.
However by the above the map $\pi_n(X_0,x)\to\pi_n(X_i,x)$ is not trivial for any $i$.
\end{example}

For completeness let us mention a version of Theorem \ref{ind-colimit} for separable 
(but not necessarily locally compact) polyhedra.

\begin{proposition} \label{ind-0} Let $L$ be a countable simplicial complex and let $K_i$ be the union of 
the first $i$ simplexes of $L$.

(a) $H_n(|L|)=0$ if and only if each $H_n(K_i)$ maps trivially to $H_n(K_j)$ for some $j\ge i$.

(b) Let $x\in K_1$ and suppose that either $n\le 1$ or $\pi_1(|L|,x)=1$.
Then $\pi_n(|L|,x)=1$ if and only if each $\pi_n(K_i,x)$ maps trivially to $H_n(K_j,x)$ for some $j\ge i$.
\end{proposition}

\begin{proof}[Proof. (a)] By using simplicial homology it is easy to see that $H_n(|L|)\simeq\colim_i H_n(K_i)$.
This immediately implies the ``if'' assertion.
Conversely, suppose that $H_n(|L|)=0$ and let us fix an $i$.
Then every element of $H_n(K_i)$ lies in the kernel $G_j$ of the map $H_n(K_i)\to H_n(K_j)$ for some $j\ge i$.
Since $K_i$ is compact, $H_n(K_i)$ is finitely generated.
Then some $G_j$ contains all its generators, and hence coincides with it.
\end{proof}

\begin{proof}[(b)] By using simplicial approximation it is easy to see that $\pi_n(|L|,x)\simeq\colim_i\pi_n(K_i,x)$.
Now the proof of the case $n\le 1$ is similar to that of (a).
To prove the case $n\ge 2$ we additionally note that if the inclusion map $(K_i,x)\to (K_j,x)$ is trivial on $\pi_1$,
then its restriction to the component of $x$ factors through a simply-connected pointed compact polyhedron,
whose $\pi_n$ is finitely generated by Serre's theorem (see \cite{Sp}*{9.6.16}).
\end{proof}

\subsection{Non-abelian examples}

\begin{example} \label{unbounded-colim}
Theorem \ref{bounded-colim}(a) does not hold for non-abelian $G_{ij}$.
Indeed, let $C_0=\Z$ and for $n\ge 1$ let $C_n$ be the coset $(1+3+\dots+3^{n-1})+3^n\Z$ of the subgroup $3^n\Z$ of $\Z$.
Clearly $C_0\supset C_1\supset\dots$ and $\bigcap_n C_n=\emptyset$.
Let $E_n=\Z\but C_n$.
Thus $\emptyset=E_0\subset E_1\subset\dots$ and $\bigcup_n E_n=\Z$.
Let $X=\R\x\{0\}\cup\Z\x D^2$ and $X_n=\R\x\{0\}\cup\Z\x S^1\cup E_n\x D^2$, where $0\in S^1=\partial D^2$.
The left regular action of $\Z$ on $\R$ yields an action of $\Z$ on $X$, and $3^n\Z$ acts on $X_n$.
Let $P_{mn}=X_n/3^m\Z$ for $m\ge n$ and $P_{mn}=X/3^m\Z$ for $m<n$.
Let $G_{mn}=\pi_1(P_{mn},x_m)$, where $x_m$ is the orbit of $(0,0)\in X_0\subset X_n\subset X$ under $3^m\Z$.
Let us note that $G_{mn}$ is a finitely generated free group.
Let $\hat X_n=\lim\big(\dots\to P_{1n}\to P_{0n}\big)$ with the obvious bonding maps, and let 
$\Gamma_n=\lim\big(\dots\to G_{1n}\to G_{0n}\big)$.
Clearly $\pi_2(P_{mn})=0$ for all $m$ and $n$, so by the Quigley short exact sequence 
(see \cite{M1}*{Theorem 3.1(b)}) the Steenrod fundamental group $\pi_1(\hat X_n,x)\simeq\Gamma_n$,
where $x=(0,0)\in X_0\subset\hat X_0\subset\hat X_1\subset\dots$. 
Since the bonding maps $P_{m+1,\,n}\to P_{mn}$ are the $3$-fold coverings 
(and hence satisfy the homotopy lifting property) for $m\ge n$, the Steenrod fundamental group 
$\pi_1(\hat X_n,x)$ coincides with the usual fundamental group $\spi_1(\hat X_n,x)$.
Since $X_n$ is the path component of $x$ in $\hat X_n$, the latter coincides with
$\spi_1(X_n,x)$, which is the countably generated free group.
The inclusions $X_0\subset X_1\subset\dots$ induce maps $\Gamma_0\to\Gamma_1\to\dots$,
and by our construction of the $X_i$ each generator, and hence each element of each $\Gamma_i$ maps 
trivially to some $\Gamma_j$, but $\Gamma_i$ does not map trivially to any $\Gamma_j$.
\end{example}

\begin{remark} \label{unbounded-colim'}
Example \ref{unbounded-colim} still makes sense with $\pi_1$, $\spi_1$ 
replaced by $\pi_k,\spi_k$ for $k\ge 2$, as long as $(D^2,S^1)$ is replaced with 
$(D^{k+1},\partial D^{k+1})$.
(Let us note that although $\pi_{k+1}(S^k)$ may be nonzero, the 3-fold covering 
$P_{m+1,\,n}\to P_{mn}$, $m\ge n$, induces an isomorphism on $\pi_i$ for all $i\ge 2$, 
so in particular $\lim^1_m\pi_{k+1}(P_{mn})=0$.)
However, for the so modified polyhedron $P_{mn}$ the group $G_{mn}=\pi_k(P_{mn})$ will 
no longer be finitely generated for $k\ge 2$ (even though it will be a finitely generated 
module over the ring $\Z[\pi_1(S^1)]=\Z[x^{\pm1}]$).
Thus for $k\ge 2$ this is no longer an example showing that ``abelian'' cannot be dropped 
in Theorem \ref{bounded-colim}(a), but rather that ``finitely generated'' cannot be dropped in it.
\end{remark}

\begin{example} \label{unbounded-colim2}
For each $k\ge 1$ there exists a local compactum $T^{k+2}$ with a basepoint $x$ such that
the Steenrod--Sitnikov homotopy group $\pi_k(T^{k+2},x)=1$, but the ind-group formed
by the Steenrod homotopy groups $\pi_k(K,x)$, where $K$ runs over all compact subsets of 
$T^{k+2}$ containing $x$, is nontrivial.

Indeed, let $P_{mn}$, $\Gamma_n$, $\hat X_n$ and $X$ be as in Example \ref{unbounded-colim}, 
but with $(D^2,S^1)$ replaced by $(D^{k+1},\partial D^{k+1})$ when $k\ge 2$.
Let $P_{m,[0,\infty)}$ be the mapping telescope of the obvious inclusions 
$P_{m0}\subset P_{m1}\subset\dots$ (which stabilize at $P_{mm}=P_{m,m+1}=\dots=X/3^m\Z$).
Clearly $T^{k+2}\bydef \lim(\dots\to P_{1,[0,\infty)}\to P_{0,[0,\infty)})$ is the mapping telescope of 
the inclusions $\hat X_0\subset\hat X_1\subset\dots$.
Let $T_n$ be the finite mapping telescope of the inclusions $\hat X_0\subset\dots\subset\hat X_n$.
Since $T_n$ deformation retracts onto $\hat X_n$, the Steenrod--Sitnikov homotopy group 
$\pi_i(T^{k+2},x_0)\simeq\colim\pi_i(T_n,x_0)\simeq\colim\pi_i(\hat X_n,x_n)$.
Hence $\pi_k(T^{k+2},x_0)\simeq\colim\Gamma_n=1$, but the homomorphism $\pi_k(T_0,x_0)\to\pi_k(T_n,x_0)$ 
is nontrivial for each $n$.
\end{example}

\begin{example} Let $T^{k+2}$ be the local compactum of Example \ref{unbounded-colim2} and let $\Sigma$ 
be the $3$-adic solenoid.
Then the inclusion of $\Sigma=\lim\big(\dots\to\R/9\Z\to\R/3\Z\to\R/\Z\big)$ in $\hat X_0\subset T^{k+2}$
clearly induces isomorphisms $\pi_i(\Sigma)\xr{\simeq}\pi_i(T^{k+2})$ on all Steenrod--Sitnikov homotopy groups, 
but is not a fine shape equivalence (since the ind-equivalence class of the 
ind-group $\pi_k(T_0)\to\pi_k(T_1)\to\dots$ is nontrivial and is an invariant of fine shape).
\end{example}

\begin{example}
Let us note a variant of Example \ref{unbounded-colim}.
Let $F$ be the free group $\left<x,y\mid\,\right>$ and $F'$ be its commutator subgroup.
Let us note that $F'$, being the fundamental group of the free abelian cover 
$L\bydef \R\x\Z\cup\Z\x\R\subset\R^2$ of $S^1\vee S^1$ (with basepoint at $(0,0)$), 
is freely generated by the set $\{[x,y]^{x^iy^j}\mid i,j\in\Z\}$, where $g^h=h^{-1}gh$.
Let $H_m$ be the subgroup $3^m\Z\oplus 3^m\Z$ of $F/F'\simeq\Z\oplus\Z$, 
and let $G_m$ be its preimage in $F$.
Writing $X$ for the universal cover of $S^1\vee S^1$, we have $X/G_m=(X/F')/(G_m/F')=L/H_m$.
Thus $G_m$ is the fundamental group of $L/H_m$, and in particular it is finitely generated.
More specifically, $G_0=F$ and $G_m$ for $m>0$ is freely generated by the set 
$\{x^{3^m},y^{3^m}\}\cup S_m$, where 
$S_m=\{[x,y]^{x^iy^j}\mid i,j\in\Z\cap[\frac{1-3^{m-1}}2,\,\frac{3^{m-1}-1}2]\}$.
Since $\bigcap_m G_m$ is the preimage of $\bigcap_m H_m=1$, it equals $F'$.

On the other hand, let $T_n=\{[x,y]^{x^iy^j}\mid i,j\in E_n\}$, where $E_n$ is as in 
Example \ref{unbounded-colim}.
Let $K_0=1$ and let $K_n$ for $n>0$ be the normal closure of $T_n$ in $F'$.
Then $\Gamma_n\bydef F'/K_n$ is the fundamental group of $P_n\bydef L\cup C_n\x C_n$, where 
$C_n=[\frac{1-3^{n-1}}2,\frac{3^{n-1}-1}2]+3^n\Z$, the union of cosets of the subgroup $3^n\Z$ of $\R$.
Clearly $K_0\subsetneq K_1\subsetneq\dots$ and $\bigcup_n K_n=F'$.
Hence $\colim\Gamma_n=1$, but no map $\Gamma_0\to\Gamma_n$ is trivial.

Finally, let us note that $T_n\subset F'\subset G_m$.
Let $K_{m0}=1$ and let $K_{mn}$ for $n>1$ be the normal closure of $T_n$ in $G_m$.
Let us note that $H_n$ acts freely on $P_n$.
Then $G_{mn}\bydef G_m/K_{mn}$ is the fundamental group of $P_{mn}$, where 
$P_{mn}=P_n/H_m$ for $m\ge n$ and $P_{mn}=\R^2/H_m$ for $m<n$.
Now $\lim\big(\dots\to G_{1n}\to G_{0n}\big)\simeq\Gamma_n$ similarly to 
Example \ref{unbounded-colim}.
\end{example}

\section{Preliminaries}

\subsection{Relative fine shape}

Given a metrizable space $X_0$ and its closed subsets $X_1\supset\dots\supset X_r$, there exist
an absolute retract $M_0$ and its closed subsets $M_1\supset\dots\supset M_r$ such that each $M_i$ 
is an absolute retract containing $X_i$ as a closed Z-set.
For instance, as long as $X_r\ne pt$, one can take each $M_i$ to be the space of probability measures $P(X_i)$
(see \cite{M00}*{Theorem \ref{book:prob}}).
Alternatively, one may construct the $M_i$ using the Fr\'echet--Wojdyslawski construction 
(see \cite{M-I}*{proof of Lemma \ref{fish:ar-embedding}}).
When $X_0$ is separable, there is also a more combinatorial construction, based on the extended mapping telescope
(see \cite{M-I}*{proof of Lemma \ref{fish:fineZ}}).
In the case where $X_1$ is compact, a number of further constructions is discussed in 
\cite{M-I}*{Remark \ref{fish:fineZ'}}.

Given a metrizable space $Y_0$ and its closed subsets $Y_1\supset\dots\supset Y_r$, there similarly exist
an absolute retract $N_0$ and its closed subsets $N_1\supset\dots\supset N_r$ such that each $N_i$ 
is an absolute containing $Y_i$ as a closed Z-set.

We recall from ``Fine shape I'' \cite{M-I} that a map $f\:M_0\but X_0\to N_0\but Y_0$ is called {\it $X_0$-$Y_0$-approaching} 
if it sends every sequence of points which has a cluster point in $X_0$ into a sequence which has a cluster point in $Y_0$.
More revealing equivalent forms of this definition are discussed in \cite{M-I}.
Clearly, if $f$ is $X$-$Y$-approaching and sends $M_i\but X_i$ into $N_i\but Y_i$, then its restriction
$f_i\:M_i\but X_i\to N_i\but Y_i$ is $X_i$-$Y_i$-approaching.

Now {\it fine shape classes} $(X_0,\dots,X_r)\to(Y_0,\dots,Y_r)$ are the classes of $X_0$-$Y_0$-approaching maps 
$(M_0\but X_0,\dots,M_r\but X_r)\to (N_0\but Y_0,\dots,N_r\but Y_r)$ up to $X_0$-$Y_0$-approaching homotopies%
\footnote{That is, homotopies that are $(X_0\x I)$-$Y_0$-approaching as maps $(M_0\but X_0)\x I\to N_0\but Y_0$.}
through maps $(M_0\but X_0,\dots,M_r\but X_r)\to (N_0\but Y_0,\dots,N_r\but Y_r)$.
Everything in \cite{M-I}*{Lemma \ref{fish:Milnor} and \S\ref{fish:fsm-section}} generalizes straightforwardly to 
the relative case, and in particular we obtain that fine shape classes $(X_0,\dots,X_r)\to(Y_0,\dots,Y_r)$, 
as defined above, do not depend on the choice of $(M_0,\dots,M_r)$ and $(N_0,\dots,N_r)$.

\subsection{Relative Steenrod--Sitnikov homotopy groups}

Let $X$ be a pointed metrizable space, $A$ be its closed subspace containing the basepoint $x$, 
and suppose that either $n>0$ or $n=0$ and $A=\{x\}$.
The pointed set $\pi_n(X,A,x)$ of fine shape classes $(I^n,\partial I^n,I^{n-1})\to (X,A,\{x\})$
is easily seen to be the direct limit $\colim\pi_n(K, K\cap A,\,x)$ over all compact $K\subset X$ 
containing $x$.
(The case where $X$ is compact is discussed in \cite{M1}.)
When $A\ne\{x\}$, we define $\pi_0(X,A)$ to be the quotient of $\pi_0(X)$ by the image of $\pi_0(A)$.
Since $\colim$ is an exact functor, $\pi_0(X,A,x)$ is again the direct limit $\colim\pi_0(K, K\cap A,\,x)$ 
over all compact $K\subset X$ containing $x$.
Clearly $\pi_i(X,x)\bydef \pi_i(X,\{x\},x)$ is a group for $i\ge 1$, which is abelian for $i\ge 2$, and 
$\pi_i(X,A,x)$ is a group for $i\ge 2$, which is abelian for $i\ge 3$.
We call these the {\it Steenrod--Sitnikov homotopy groups.}
We will sometimes omit the basepoint from the notation.
There is an action of $\pi_1(A)$ on $\pi_n(X)$ for $n\ge 1$ and on $\pi_n(X,A)$
for $n\ge 2$, and an action of $\pi_1(X)$ on $\pi_1(X,A)$.

There is an exact sequence of pointed sets
\[\dots\to\pi_2(X,A)\to\pi_1(A)\to\pi_1(X)\to\pi_1(X,A)\to\pi_0(A)\to\pi_0(X)\to\pi_0(X,A)\to *,\]
whose maps to the left of $\pi_1(X)$ (resp.\ $\pi_2(X)$) are homomorphisms
of groups (resp.\ right $\Z\pi_1(A)$-modules).
In addition, $\pi_2(X)\to\pi_2(X,A)$ is also $\pi_1(A)$-equivariant,
$\partial\:\pi_2(X,A)\to\pi_1(A)$ is a crossed module%
\footnote{That is, $\partial(s\cdot g)=g^{-1}(\partial s)g$ for $g\in\pi_1(A)$,
$s\in\pi_2(X,A)$ and $s^{-1}ts=t\cdot(\partial s)$ for $s,t\in\pi_2(X,A)$.}%
, $\pi_1(X)\to\pi_1(X,A)$ is $\pi_1(X)$-equivariant with respect to the right
regular action on $\pi_1(X)$, and the non-trivial point-inverses of
$\pi_1(X,A)\to\pi_0(A)$ coincide with the orbits of $\pi_1(X)$.

\subsection{Base-ray preserving approaching homotopy classes}

In general, not every element of $\pi_0(X)$ is represented by a map $pt\to X$ 
(see \cite{M1}*{Theorem 8.4 or Example 8.6}).
Because of this, it is sometimes convenient to replace basepoints by base rays.

Let $X$ be a metrizable space and let $M$ be an absolute retract containing $X$
as a closed Z-set.
Let $r\:[0,\infty)\to M\but X$ be an $\{\infty\}$-$X$-approaching map; in other words,
any map such that the closure $\overline{r\big([0,\infty)\big)}=r\big([0,\infty)\big)\cup K$
for some non-empty compact $K\subset X$. 
For a polyhedron $P$ with a basepoint $x$, a $P$-$X$-approaching map $P\x[0,\infty)\to M\but X$ 
will be called {\it $r$-pointed} if its restriction $\{x\}\x[0,\infty)\to M\but X$
coincides with $r$.

We write $\Pi_n(M\but X,\,r)$ for the set of classes of $r$-pointed $S^n$-$X$-approaching maps
$S^n\x[0,\infty)\to M\but X$ up to $r$-pointed $S^n$-$X$-approaching homotopy.
If $A$ is a closed subset of $X$, contained as a Z-set in a closed absolute retract $N\subset M$,
and $r\:[0,\infty)\to N\but A$ is an $\{\infty\}$-$A$-approaching map, then we write
$\Pi_n(M\but X,\,N\but A,\,r)$ for the set of classes of $r$-pointed $I^n$-$X$-approaching maps
$(I^n,\partial I^n)\x[0,\infty)\to (M\but X,\,N\but A)$ up to $r$-pointed $I^n$-$X$-approaching homotopy
(the basepoint of $I^n$ is assumed to lie in $\partial I^n$).

All that has been said about $\pi_n(X,x)$ and $\pi_n(X,A,x)$ generalizes straightforwardly to 
$\Pi_n(M\but X,\,r)$ and $\Pi_n(M\but X,\,N\but A,\,r)$.
In particular, $\Pi_n(M\but X,\,r)$ is a group for $i\ge 1$, which is abelian for $i\ge 2$, and 
$\Pi_n(M\but X,\,N\but A,\,r)$ is a group for $i\ge 2$, which is abelian for $i\ge 3$.

\subsection{\v Cech homotopy groups}

For a compactum $K$ with a basepoint $x$, the {\it \v Cech homotopy set} $\check\pi_n(K,x)$ is the inverse limit 
$\lim\pi_n(P_i,x_i)$, where $\dots\to (P_1,x_1)\to (P_0,x_0)$ is an inverse sequence of pointed compact polyhedra 
with $\lim (P_i,x_i)=(K,x)$.
It is a group for $n\ge 1$.
It is a well defined shape invariant of $(K,x)$ (see \cite{M00}*{Theorem \ref{book:fun-extension}}), and hence also
a strong shape invariant of $(K,x)$.

\begin{lemma} \label{continua}
Let $X$ be a locally connected compactum and $x\in X$.
Then $\check\pi_0(X,x)$ is finite and the map $\pi_0(X,x)\to\check\pi_0(X,x)$ is a bijection.
\end{lemma}

\begin{proof}
Since $X$ is locally connected, the map $\pi_0(X,x)\to\check\pi_0(X,x)$ is a bijection and $\check\pi_0(X,x)$ 
is discrete when endowed with the topology of the inverse limit of discrete sets \cite{M1}*{Theorem 6.1(c,g)}.%
\footnote{There is a minor gap in \cite{M1}*{proofs of Lemma 6.2(1,3) and Theorem 6.1(d,h,g) in the case $n=0$} 
as it was overlooked there that for a map between pointed sets trivial kernel does not imply injectivity.
This gap can be filled in by trivial modifications of the arguments.
In particular, when \cite{M1}*{Theorem 3.1(b)} is applied in the proof of \cite{M1}*{Lemma 6.2(1)}, the base ray 
in the former can be arbitrary, and has nothing to do with the base point in the latter.}
If $\dots\to (P_1,x_1)\to (P_0,x_0)$ is an inverse sequence of pointed compact polyhedra with $\lim (P_i,x_i)=(X,x)$,
then $\check\pi_0(X,x)$ is isomorphic to $\lim\pi_0(P_i,x_i)$.
Since it is discrete, it injects in $\pi_0(P_k,x_k)$ for some $k$ \cite{M1}*{Lemma 3.4(b)}.
Hence it is finite.
\end{proof}

For a metrizable space $X$ with a basepoint $x$ let $\breve\pi_n(X,x)$ be the direct limit of the \v Cech homotopy set 
$\check\pi_n(K,x)$ over all compact subsets $K$ of $X$ containing the basepoint.

On the other hand, by $\spi_n(X,x)$ we will denote the {\it usual} homotopy set $[(S^n,pt),\,(X,x)]$.
Both $\breve\pi_n(X,x)$ and $\spi_n(X,x)$ are antishape invariants of $(X,x)$ (see \cite{M00}*{Theorem \ref{book:sit-extension}}), 
and hence also fine shape invariants of $(X,x)$.

\begin{lemma} \label{Cech}
Let $X$ be a metrizable space with a basepoint $x$.

(a) If $\pi_n(X,x)=1$, then $\breve\pi_n(X,x)=1$.

(b) If $X$ is a locally $(n-1)$-connected local compactum and each compactum $L\subset X$ containing the basepoint 
lies in a compactum $M\subset X$ such that the map $\spi_n(L,x)\to\spi_n(M,x)$ is trivial, 
then $\breve\pi_n(X,x)=1$.
\end{lemma}

\begin{proof}[Proof. (a)]
Since direct limit is an exact functor, the surjections $\pi_n(K,x)\to\check\pi_n(K,x)$ (see \cite{M1}*{Theorem 3.1(b)})
yield a surjection $\pi_n(X,x)\to\breve\pi_n(X,x)$.
\end{proof}

\begin{proof}[(b)] Given a compactum $K\subset X$ containing the basepoint, since $X$ is locally compact, 
there exists a compactum $L\subset X$ such that $K\subset\Int L$.
Then the inclusion $K\to L$ factors through a locally connected compactum $Q$ \cite{M1}*{Theorem 3.1(b)}.
Since $Q$ is locally connected, the image $D$ of the map $\spi_n(Q,x)\to\check\pi_n(Q,x)$ is dense 
in the topology of the inverse limit of discrete sets (see \cite{M1}*{Theorem 6.1(e)}).
On the other hand, by the hypothesis $L$ lies in a compactum $M$ such that the map $\spi_n(L,x)\to\spi_n(M,x)$ 
is trivial.
Then in particular the image of $D$ in $\check\pi_n(M,x)$ is trivial.
But $D$ is dense in the image $E$ of $\check\pi_n(Q,x)$ in $\check\pi_n(M)$.
Hence $E=1$.
Therefore the map $\check\pi_n(K,x)\to\check\pi_n(M,x)$ is trivial.
Thus $\breve\pi_n(X,x)=1$.
\end{proof}

\section{Auxiliary constructions}

\subsection{Straightening construction}

There is an almost obvious trick that converts a direct sequence of inv-morphisms between inverse sequences into a direct sequence 
of level maps by repeating some terms of the second inverse sequence, then some terms of the third inverse sequence, and so on.
In more detail:

\begin{lemma} \label{straightening}
Let $T_0\xr{f_0}T_1\xr{f_1}\dots$ be a direct sequence of inv-morphisms between inverse sequences 
$T_j=\big(\dots\xr{p_{1j}} G_{1j}\xr{p_{0j}} G_{0j}\big)$ in some category, with 
$f_j=\big(f_{ij}\:G_{l_{ij},j}\to G_{i,j+1}\big)$.
Then there exists a direct sequence of level maps $T_0'\xr{g_0}T'_1\xr{g_1}\dots$ between 
inverse sequences of the form 
$T'_i=\big(\dots\xr{p_{0j}} G_{1j}\xr{\id}\dots\xr{\id} G_{1j}\xr{p_{0j}} G_{0j}\xr{\id}\dots\xr{\id} G_{0j}\big)$
such that each $g_j$ consists of compositions of the form 
$G_{kj}\xr{p_{lj}\cdots p_{k-1,j}}G_{lj}\xr{f_{ij}}G_{i,j+1}$.
\end{lemma}

\begin{proof} We may assume without loss of generality that each $l_{i+1,j}>l_{ij}$
(see \cite{M00}*{proof of Proposition \ref{book:pro-sequence}}).
Let $i_{kj}$ be defined by $i_{k0}=k$ and by setting $i_{k,j+1}$ to be the maximal number $i$ 
such that $l_{ij}\le i_{kj}$.
It is easy to see that $i_{k+1,j}\ge i_{kj}$ and that $i_{kj}\to\infty$ as $k\to\infty$.
Let $H_{kj}=G_{i_{kj},j}$.
Let $g_{ij}\:H_{kj}\to H_{k,j+1}$ be the composition 
\[H_{kj}=G_{i_{kj},j}\xr{p'_{ij}} G_{l_{ij},j}\xr{f_{ij}} G_{i,j+1}=H_{k,j+1},\]
where $i=i_{k,j+1}$ and $p'_{ij}$ is the bonding map of $T_j$.
Let $T'_j=\big(\dots\xr{q_{1j}} H_{1i}\xr{q_{0j}} H_{0i}\big)$, where each 
$q_{ij}$ is the bonding map of $T_j$.
Then the level map $g_j\:T'_j\to T'_{j+1}$ consisting of the $g_{ij}$ is as desired.
\end{proof}

\subsection{Two-parameter mapping cylinder}

Speaking informally, the following lemma is saying that given a homotopy commutative square diagram,
the two versions of its two-parameter mapping cylinder (that is, the two obvious ways of ``filling in the square'' 
whose boundary is formed by the mapping cylinders of the 4 given maps) are homeomorphic.

\begin{lemma} \label{2-mc}
Given a homotopy commutative diagram of compacta
\[\begin{tikzcd}[row sep=0.8em,column sep=1.4em]
& Y_+ \ar[rd,"g_+"] & \\
X \ar[ru, "f_+"]\ar[rd, "f_-"']  & & Z, \\
& Y_- \ar[ru,"g_-"'] &         
\end{tikzcd}\]
let $\Phi_\pm$ be the composition $MC(f_\pm)\cong X\x I\cup MC(f_\pm)\xr{\pi_\mp\cup\rho_\pm} MC(g_\mp)\cup Z=MC(g_\mp)$, where 
$\pi_\mp$ is the composition $X\x I\xr{f_\mp\x\id_I}Y_\mp\x I\xr{q}MC(g_\mp)$ and 
$\rho_\pm$ is defined by the strictly commutative diagram
\[\begin{tikzcd}
& MC(f_\pm) \ar[rd,"\rho_\pm"] & \\
X\x I \ar[ru,"q"]\ar[rr,"h_\pm"]  & & Z,
\end{tikzcd}\]
where $h_+$ is the given homotopy and $h_-(x,t)=h_+(x,1-t)$.
Then $MC(\Phi_+)\cong MC(\Phi_-)$.
\end{lemma}

\begin{proof} Let $W=MC(g_+)\cup_Z MC(g_-)$.
Let $\phi\:X\x[0,3]\to W$ be the union of the compositions 
\begin{align*}
X\x[0,1]&\xr{f_+\x(t\,\mapsto\, 1-t)}Y_+\x I\xr{q}MC(g_+)\\
X\x[1,2]&\xr{\id\x(t\,\mapsto\, t-1)}X\x I\xr{h}Z\\
X\x[2,3]&\xr{f_-\x(t\,\mapsto\, t-2)}Y_-\x I\xr{q}MC(g_-).
\end{align*}
Let $e_+\:[0,3]\to I\x I$ be defined by $e_+(t)=(1-t,0)$ for $t\in [0,1]$ and by $e_+(t)=(0,\frac{t-1}2)$ for $t\in [1,3]$.
Let $e_-$ be the composition $[0,3]\xr{e_+}I\x I\xr{(t,s)\mapsto\,(s,t)}I\x I$.
Finally, let $\psi_\pm$ be the composition $X\x(I\x\{0\}\cup\{0\}\x I)\xr{\id_X\x e_\pm^{-1}}X\x [0,3]\xr{\phi} W$.

Then it is easy to see that $MC(\Phi_\pm)$ is the adjunction space $X\x I\x I\cup_{\psi_\pm}W$.
Now the assertion follows since $e_+$ and $e_-$ differ by a homeomorphism.
\end{proof}

\subsection{Mapping cylinder lemma}

The following lemma will be used in the proof of Theorem \ref{ind-isomorphism2}.
It is not needed for Theorem \ref{Wh-thm}.

\begin{lemma} \label{ind-isomorphism}
Let $X_0\subset X_1\subset\dots$ and $Y_0\subset Y_1\subset\dots$ be 
metrizable spaces, $f_i\:X_i\to Y_i$ be maps such that each $f_{i+1}$ is
an extension of $f_i$, and let $M_i=MC(f_i)$.

(a) The induced maps $f_{i*}\:H_n(X_i)\to H_n(Y_i)$ represent an ind-isomorphism for each $n$
if and only if for each $n$ and $j$ there exists a $k\ge j$ such that the inclusion induced map 
$H_n(M_j,X_j)\to H_n(M_k,X_k)$ is trivial.

(b) Suppose that for each $j$ and each $x\in X_j$, all elements of $\pi_0(X_{j+1},x)$
that come from $\pi_0(X_j,x)$ are represented by points of $X_{j+1}$.
Then the following are equivalent:
\begin{enumerate}
\item for each $n$ and $i$ there exists a $k\ge i$ such that for each $x\in X_i$ there exists 
a map $\phi\:\pi_n\big(Y_i,f(x)\big)\to\pi_n(X_k,x)$ such that the compositions 
$\pi_n(X_i,x)\xr{f_{i*}}\pi_n\big(Y_i,f(x)\big)\xr{\phi}\pi_n(X_k,x)$ and 
$\pi_n\big(Y_i,f(x)\big)\xr{\phi}\pi_n(X_k,x)\xr{f_{k*}}\pi_n\big(Y_k,f(x)\big)$
coincide with the bonding maps;
\item for each $n$ and $i$ there exists a $k\ge i$ such that for each $x\in X_i$ the inclusion 
induced map $\pi_n(M_i,X_i,x)\to\pi_n(M_k,X_k,x)$ is trivial.
\end{enumerate}
\end{lemma}

\begin{proof}[Proof. (a)] 
Suppose that the induced maps $f_{i*}\:H_n(X_i)\to H_n(Y_i)$ represent an ind-isomorphism for some $n$.
Then its inverse can be represented by maps of the form $\phi_i\:H_n(Y_i)\to H_n(X_{\alpha_n(i)})$ 
for some strictly monotone map $\alpha_n\:\N\to\N$ (see \cite{M00}*{Proposition \ref{book:ind-sequence}}).
Then the compositions $H_n(X_i)\xr{f_{i*}}H_n(Y_i)\xr{\phi_i}H_n(X_{\alpha_n(i)})$ and
$H_n(Y_i)\xr{\phi_i}H_n(X_{\alpha_n(i)})\xr{f_{\alpha_n(i)*}}H_n(Y_{\alpha_n(i)})$
represent the identity ind-morphisms.
We may assume that these compositions equal the bonding maps, by replacing each $\phi_i$ with 
the composition $H_n(Y_i)\xr{\phi_i}H_n(X_{\alpha_n(i)})\to H_n(X_{\alpha'_n(i)})$
for an appropriate $\alpha'_n(i)\ge\alpha_n(i)$.

Now the proof is similar to the following proof of (b), but easier.
\end{proof}

\begin{proof}[(b)] 
{\it (1)$\Rightarrow$(2).} We are given an $n$ and an $i$.
Then there exist a $j\ge i$ and a $k\ge j$ such that
for each $x\in X_i$ there exist maps 
$\phi\:\pi_{n-1}\big(Y_i,f(x)\big)\to\pi_{n-1}(X_j,x)$ and
$\psi\:\pi_n\big(Y_j,f(x)\big)\to\pi_n(X_k,x)$
such that the compositions 
$\pi_{n-1}(X_i,x)\xr{f_{i*}}\pi_{n-1}\big(Y_i,f(x)\big)\xr{\phi}\pi_{n-1}(X_j,x)$ and 
$\pi_n\big(Y_j,f(x)\big)\xr{\psi}\pi_n(X_k,x)\xr{f_{k*}}\pi_n\big(Y_k,f(x)\big)$
coincide with the bonding maps.
Now with this choice of $k$ we are also given an $x\in X_i$.
Then a straightforward diagram chasing in the following commutative diagram with exact rows,
where the basepoints $x$ and $f(x)$ are suppressed:
\[\begin{tikzcd}[row sep=1.5em,column sep=1.5em]
& & & {\pi_n(M_i,X_i)} \rar \dar & \pi_{n-1}(X_i) \ar[rr,"f_{i*}"] \dar & & \pi_{n-1}(Y_i) \ar[dll,"\phi"] \\
& & \pi_n(Y_j) \rar \dar \ar[dll,"\psi"'] & \pi_n(M_j,X_j) \rar \dar & \pi_{n-1}(X_j) & & \\
\pi_n(X_k) \ar[rr,"f_{k*}"'] & & \pi_n(Y_k) \rar & \pi_n(M_k,X_k) & &
\end{tikzcd}\]
shows that the bonding map $\pi_n(M_i,X_i)\to\pi_n(M_k,X_k)$ is trivial.
(In the case $n=0$, the same diagram makes sense if we set $\pi_{-1}$ to always be the trivial pointed set.)

{\it (2)$\Rightarrow$(1).}
We are given an $n$ and an $i$.

Let us first consider the case $n\ge 1$.
There exist a $j\ge i$ and a $k\ge j$ such that the inclusion induced maps
$\pi_n(M_i,X_i,x)\to\pi_n(M_j,X_j,x)$ and $\pi_{n+1}(M_j,X_j,x)\to\pi_{n+1}(M_k,X_k,x)$ 
are trivial for each $x\in X_i$.
Now with this choice of $k$ we are also given an $x\in X_i$.
We will use the following commutative diagram with exact rows, where the basepoints $x$ and $f(x)$ are suppressed:
\[\begin{tikzcd}
& & \pi_n(Y_i) \rar\dar & \pi_n(M_i,X_i)\dar["1"] \\
\pi_{n+1}(M_j,X_j) \rar\dar["1"] & \pi_n(X_j)\rar["f_{j*}"]\dar & \pi_n(Y_j) \rar & \pi_n(M_j,X_j) \\
\pi_{n+1}(M_k,X_k)\rar & \pi_n(X_k). &
\end{tikzcd}\]
Given a $g\in\pi_n(Y_i)$, its image in $\pi_n(Y_j)$ equals the image of some $h\in\pi_n(X_j)$.
Given another element $h'\in\pi_n(X_j)$ whose image in $\pi_n(Y_j)$ equals the image of $g$,
it is easy to see that $h^{-1}h'$ has trivial image in $\pi_n(X_k)$.
Hence the image of $h$ in $\pi_n(X_k)$, denoted $\phi(g)$, depends only on $g$.
Given a $g'\in\pi_n(Y_i)$, it is easy to see that 
$\phi(gg')^{-1}\phi(g)\phi(g')=1$.
Hence $\phi\:\pi_n\big(Y_i,f(x)\big)\to\pi_n\big(X_k,x\big)$ is a homomorphism.
By construction the compositions 
$\pi_n(X_i)\xr{f_{i*}}\pi_n(Y_i)\xr{\phi}\pi_n(X_k)$ and
$\pi_n(Y_i)\xr{\phi}\pi_n(X_k)\xr{f_{k*}}\pi_n(Y_k)$ equal the bonding maps.

It remains to consider the case $n=0$.
There exist a $j\ge i$ and a $k\ge j$ such that the inclusion induced maps
$\pi_0(M_i,X_i,x)\to\pi_0(M_j,X_j,x)$ and $\pi_1(M_{j+1},X_{j+1},z)\to\pi_1(M_k,X_k,z)$ 
are trivial for each $x\in X_i$ and for each $z\in X_j$.
Now with this choice of $k$ we are also given an $x\in X_i$.
Let $y=f(x)$.
Given a $g\in\pi_0\big(Y_i,y\big)$, its image in $\pi_0\big(Y_j,y)$ equals the image of 
some $h\in\pi_0(X_j,x)$.
Suppose that it has the same image in $\pi_0\big(Y_j,y\big)$ as some other element $h'\in\pi_0(X_j,x)$.
By the hypothesis the images of $h$ and $h'$ in $\pi_0(X_{j+1},x)$
are represented by some points $z,z'\in X_{j+1}$.
Then the class of $z'$ in $\pi_0(X_{j+1},z)$ maps to the trivial element of $\pi_0\big(Y_{j+1},f(z)\big)$.
Hence it is the image of some $a\in\pi_1(M_{j+1},X_{j+1},z)$.
By our choice of $k$ this $a$ maps trivially to $\pi_1(M_k,X_k,z)$.
Hence $z'$ represents the trivial element of $\pi_0(X_k,z)$.
Therefore $z$ and $z'$ represent the same element $b$ of $\pi_0(X_k,x)$.
Thus this $b$ depends only on $g$.
Then the argument proceeds as before.
\end{proof}

\begin{remark} \label{coherent}
Let us formulate a version of Lemma \ref{ind-isomorphism}(b) where we are not given
genuine maps $f_i\:X_i\to Y_i$, but merely a coherent map (in the sense of sequential strong antishape,
see \cite{M-I}*{\S\ref{fish:sash}}); on the other hand, we shall not need to assume that
every element of $\pi_0(X_i,x)$ is represented by a map $pt\to X_i$.

Thus let $X$ and $Y$ be local compacta.
We are given metrizable spaces $U$ and $V$ containing respectively $X$ and $Y$,
and chains of subsets $U_0\subset U_1\subset\dots$ of $U$ and
of subsets $V_0\subset V_1\subset\dots$ of $V$ such that
\begin{enumerate}
\item each $U_i$ and $V_i$ are compact absolute retracts;
\item each $X_i\bydef U_i\cap X$ is a Z-set in $U_i$ and each $Y_i\bydef V_i\cap Y$ is a Z-set in $V_i$;
\item $\{X_0,X_1,\dots\}$ and $\{Y_0,Y_1,\dots\}$ are cofinal subsets of $\K_X$ and $\K_Y$,
\end{enumerate}
where $\K_X$ is the poset of all compact subsets of $X$, ordered by inclusion.
Writing $U_i^\circ=U_i\but X_i$ and $V_i^\circ=V_i\but Y_i$, we are also given
proper maps $f_i\:U_i^\circ\to V_i^\circ$ along with proper homotopies 
$F_i\:U_i^\circ\x I\to V_{i+1}^\circ$ between the two compositions of  the diagram
\[\begin{tikzcd}[row sep=2em]
U_{i+1}^\circ\ar[r,"f_{i+1}"]&V_{i+1}^\circ\\
U_i^\circ\ar[r,"f_i"']\ar[u,hook]&V_i^\circ.\ar[u,hook]
\ar[ul,phantom,"{}_{F_i}\hspace{-2pt}\rotatebox{135}{$\Leftarrow$}"]
\end{tikzcd}\]
Next we consider the metric mapping cylinders $N_i^\circ\bydef MC(f_i)$ (see \cite{M00}*{\S\ref{book:mmc}}).
 
Then the following are equivalent:
\begin{enumerate}
\item for each $n$ and $i$ there exists a $k\ge i$ such that for every proper map 
$r\:[0,\infty)\to U_i^\circ$ there exists a map $\phi\:\Pi_n\big(V_i^\circ,f_ir\big)\to\Pi_n(U_k^\circ,r)$ 
such that the compositions $\Pi_n(U_i^\circ,r)\xr{f_{i*}}\Pi_n\big(V_i^\circ,f_ir\big)\xr{\phi}\Pi_n(U_k^\circ,r)$ and 
$\Pi_n\big(V_i^\circ,f_ir\big)\xr{\phi}\Pi_n(U_k^\circ,r)\xr{f_{k*}}\Pi_n\big(V_k^\circ,f_kr\big)$
coincide with the bonding maps;
\item for each $n$ and $i$ there exists a $k\ge i$ such that for for every proper map 
$r\:[0,\infty)\to U_i^\circ$ the inclusion 
induced map $\pi_n(N_i^\circ,U_i^\circ,r)\to\pi_n(N_k^\circ,U_k^\circ,r)$ is trivial.
\end{enumerate}

The proof is similar to that of Lemma \ref{ind-isomorphism}(b).
\end{remark}

\section{Factorization lemmas}

This section contains the proof of Theorem \ref{factorization-main} and some related results.

\subsection{The abelian case}

\begin{lemma} \label{level-factorization} 
Let $f=(f_i\:A_i\to C_i)_{i\in\N}$ be a level map between towers of abelian groups.

(a) $f$ factors through a tower $(B_i)_{i\in\N}$ of quotient groups $B_i$ of $A_i$ such that 
the map $\lim A_i\to\lim B_i$ is surjective and the map $\lim B_i\to\lim C_i$ is injective.
Also, the map $\lim^1 A_i\to\lim^1 B_i$ is an isomorphism.

(b) $f$ factors through a tower of subgroups $B_i\subset C_i$ such that 
the map $\lim^1 A_i\to\lim^1 B_i$ is surjective and the map $\lim^1 B_i\to\lim^1 C_i$ is injective.
Also, the map $\lim B_i\to\lim C_i$ is an isomorphism.
\end{lemma}

\begin{proof}[Proof. (a)] Let $K_i=\ker f_i$.
Let $L_i$ be the image of $\lim K_j$ in $K_i$.
Then the inclusions $L_i\to K_i$ induce an isomorphism $\lim L_i\to\lim K_i$.
Also the bonding maps $L_{i+1}\to L_i$ are surjective, so $\lim^1 L_i=0$.
Writing $Q_i=K_i/L_i$, from the short exact sequence of towers $0\to L_i\to K_i\to Q_i\to 0$
we get that $\lim Q_i=0$.
Let $B_i=A_i/L_i$.
From the short exact sequence of towers $0\to L_i\to A_i\to B_i\to 0$
we get that the map $\lim A_i\to\lim B_i$ is surjective and the map $\lim^1 A_i\to\lim^1 B_i$ is an isomorphism.
Let $G_i=f_i(A_i)$.
Then $G_i\simeq A_i/K_i\simeq (A_i/L_i)/(K_i/L_i)=B_i/Q_i$.
From the short exact sequence of towers $0\to Q_i\to B_i\to G_i\to 0$
we get that the map $\lim B_i\to\lim G_i$ is injective.
Also the map $\lim G_i\to\lim C_i$ is injective since $G_i\subset C_i$.
\end{proof}

\begin{proof}[(b)] Let $G_i=f_i(A_i)$ and let $K_i=C_i/G_i$.
Let $L_i$ be the image of $\lim K_j$ in $K_i$.
Then the inclusions $L_i\to K_i$ induce an isomorphism $\lim L_i\to\lim K_i$.
Also the bonding maps $L_{i+1}\to L_i$ are surjective, so $\lim^1 L_i=0$.
Writing $Q_i=K_i/L_i$, from the short exact sequence of towers $0\to L_i\to K_i\to Q_i\to 0$
we get that $\lim Q_i=0$.
Let $B_i$ be the preimage of $L_i$ under the quotient map $C_i\to K_i$.
Then $L_i\simeq B_i/G_i$ and consequently $C_i/B_i\simeq (C_i/G_i)/(B_i/G_i)\simeq K_i/L_i=Q_i$.
The short exact sequence of towers $0\to B_i\to C_i\to Q_i\to 0$ yields that the map 
$\lim B_i\to\lim C_i$ is an isomorphism and the map $\lim^1 B_i\to\lim^1 C_i$ is injective.
Finally, the map $\lim^1 A_i\to\lim^1 G_i$ is surjective since $G_i=f_i(A_i)$, and from 
the short exact sequence of towers $0\to G_i\to B_i\to L_i\to 0$ the map
$\lim^1 G_i\to\lim^1 B_i$ is also surjective.
\end{proof}

\begin{proof}[Proof of Theorem \ref{factorization-main}] 
By Lemma \ref{level-factorization}(a) the level map $B_i\to C_i$ factors through a tower of abelian
groups $H_i$ such that $\lim H_i=0$. 
By Lemma \ref{level-factorization}(b) the composition $A_i\to B_i\to H_i$ factors through a tower of abelian
groups $G_i$ such that $\lim^1 G_i=0$ and the map $\lim G_i\to\lim H_i$ is an isomorphism.
\end{proof}

\begin{remark} It might seem that one can prove Lemma \ref{level-factorization}(b) by
letting each $B_i$ be the purification $\{g\in C_i\mid \exists n\in\mathbb Z: ng\in f_i(A_i)\}$
of $f_i(A_i)$, but this is not so.
Indeed, there exists an inverse sequence of embeddings $\dots\to C_1\to C_0$ between free abelian groups 
of rank 2, and an inverse sequence $\dots\to B_1\to B_0$ of their pure subgroups of rank $1$ 
such that the map $\lim^1 B_i\to\lim^1 C_i$ is not injective \cite{M00}*{Example \ref{book:lattice2}}.
(A subgroup $H$ of a group $G$ is called {\it pure} if for any $g\in G$ and $n\in\Z$ such that $ng\in H$
one has $g\in H$.)
\end{remark}

\subsection{Non-abelian and non-group cases}

\begin{remark} \label{level-factorization2}
In trying to prove the non-abelian version of Lemma \ref{level-factorization}(b) we may define 
the group $G_i$ and the pointed sets of left cosets $K_i$, $L_i$ and $B_i$ as in the proof 
of the abelian case; then the map $\lim^1 A_i\to\lim^1 G_i$ is still surjective.
{\it If each $B_i$ happens to be a subgroup of $C_i$,} then the map $\lim^1 G_i\to\lim^1 B_i$ is surjective
(see \cite{M00}*{Addendum \ref{book:6-term4}(a)}), and hence there is the following commutative diagram 
with exact rows (see \cite{M00}*{Theorem \ref{book:6-term3}}):
\[\begin{tikzcd}
1\rar &\lim G_i\rar\dar[equal] &\lim B_i\rar\dar[tail] &\lim L_i\rar\dar[tail, two heads] &\lim^1 G_i\rar\dar[equal] &\lim^1 B_i\rar\dar &1\\
1\rar &\lim G_i\rar &\lim C_i\rar &\lim K_i\rar &\lim^1 G_i\rar &\lim^1 C_i
\end{tikzcd}\]
Then by an easy diagram chasing (or by a version of the 5-lemma) $\lim^1 B_i\to\lim^1 C_i$ is injective 
and $\lim B_i\to\lim C_i$ is surjective.
\end{remark}

\begin{lemma} \label{level-factorization3} 
Let $f=(f_i\:A_i\to C_i)_{i\in\N}$ be a level map between towers of groups where the tower 
$(A_i)_{i\in\N}$ satisfies the Mittag-Leffler condition.
Then $f$ factors through a tower of quotient groups $B_i$ of $A_i$ such that the map 
$\lim A_i\to\lim B_i$ is surjective and the map $\lim B_i\to\lim C_i$ is injective.
Also, $\lim^1 B_i$ is trivial.
\end{lemma}

\begin{proof} Let $K_i=\ker f_i$.
Let $L_i$ be the image of $\lim K_j$ in $K_i$.
Then the inclusions $L_i\to K_i$ induce an isomorphism $\lim L_i\to\lim K_i$.
Let $N_i$ be the normal closure of $L_i$ in $A_i$.
It is easy to see that the bonding map $A_{i+1}\to A_i$ sends $N_{i+1}$ into $N_i$.
Since $K_i$ is normal in $A_i$, we have $N_i\subset K_i$.
Since $L_i$ and $A_i$ satisfy the Mittag-Leffler condition, it is easy to see that so does $N_i$.
Hence $\lim^1 N_i=1$.
The inclusions $L_i\subset N_i\subset K_i$ induce monomorphisms 
$\lim L_i\to\lim N_i\to\lim K_i$, which must be isomorphisms since so is their composition.

Let $B_i=A_i/N_i$.
From the short exact sequence of towers $1\to N_i\to A_i\to B_i\to 1$ we get that the map
$\lim A_i\to\lim B_i$ is surjective and that $\lim^1 B_i$ is trivial.
Let $Q_i=K_i/N_i$.
From the short exact sequence of towers $1\to N_i\to K_i\to Q_i\to 1$ we get that $\lim Q_i=1$.
Since $N_i$ is normal in $A_i$, it is normal in $K_i$, and so $Q_i$ is a group.
Since $N_i$ and $K_i$ are normal in $A_i$, it is easy to see that $Q_i$ is normal in $B_i$.
Let $G_i=f_i(A_i)$, and let us note that $G_i\simeq A_i/K_i\simeq (A_i/N_i)/(K_i/N_i)=B_i/Q_i$.
From the short exact sequence of towers $1\to Q_i\to B_i\to G_i\to 1$
we get that the map $\lim B_i\to\lim G_i$ is injective.
Also the map $\lim G_i\to\lim C_i$ is injective since $G_i\subset C_i$.
\end{proof}

\begin{lemma} \label{level-factorization4} 
Let $(f_i\:A_i\to C_i)_{i\in\N}$ be a level map between towers of finite sets, and let $B_i=f_i(A_i)$.
Then the map $\lim A_i\to\lim B_i$ is surjective and the map $\lim B_i\to\lim C_i$ is injective.
\end{lemma}

\begin{proof}
The map $\lim B_i\to\lim C_i$ is injective since each map $B_i\to C_i$ is injective.
It remains to show that the map $\lim A_i\to\lim B_i$ is surjective.
Given a thread $(b_i)\in\lim B_i$, each $K_i\bydef f_i^{-1}(b_i)$ is nonempty.
Since the sets $K_i$ are finite, the tower $(K_i)$ satisfies the Mittag-Leffler condition.
Thus for each $i$ there exists a $j_i$ such that the image $K'_i$ of $K_j$ in $K_i$ does not 
depend on $j$ for all $j\ge j_i$.
If $j=\max(j_i,j_{i+1})$, then $K'_i=\im(K_j\to K_i)=\im\big(\im(K_j\to K_{i+1})\to K_i\big)=\im(K'_{i+1}\to K_i)$.
Thus the map $K'_{i+1}\to K'_i$ is surjective.
Pick any $a_0\in K'_0$.
Then pick any $a_1$ in the preimage of $a_0$ in $K'_1$.
Continuing in this fashion, we obtain a thread $(a_i)\in\lim K'_i$.
Then $(a_i)\in\lim A_i$ and $f_i(a_i)=b_i$ for each $i$.
\end{proof}

\section{Ind-groups versus colimit}

This section contains the proof of Theorem \ref{bounded-colim}.

\subsection{The case of lim}

\begin{proof}[Proof of Theorem \ref{bounded-colim}(a)] We have $\Gamma_j=\lim_i G_{ij}$.
Suppose that $\colim\Gamma_j=0$, but there exists a $k$ such that $\Gamma_k$
has a nontrivial image in each $\Gamma_j$.
Upon deleting some columns we may assume that $k=0$.
Let $K_j=\ker(\Gamma_0\to\Gamma_j)$.
Then $\Gamma_0=\bigcup_{j=1}^\infty K_j$, but $\Gamma_0\ne K_j$ for any $j$.
Let $H_i=\im(\Gamma_0\to G_{i0})$.
For any fixed $i$ we have $H_i=\bigcup_{j=1}^\infty L_{ij}$, where $L_{ij}$
is the image of $K_j$ in $G_{i0}$.
Since $G_{i0}$ is finitely generated and abelian, so is $H_i$.
Hence all generators of $H_i$ lie in some $L_{ij}$.%
\footnote{Let us note that this cannot be achieved in the non-abelian case.
Indeed, the Baumslag--Solitar group $\left<x,t\mid t^{-1}x^2t=x^3\right>$ is non-Hopfian (see \cite{MKS}*{\S4.4, p.~260})
and consequently (see \cite{MKS}*{Exercise 2.4.18, p.~114}) contains a strictly ascending chain of normal subgroups 
$K_0\subsetneq K_1\subsetneq\dots$.}
Thus for each $i$ there exists a $j$ such that $K_j$ maps onto $H_i$.

Let $g^1\in\Gamma_0\but K_1$.
Since the image $h^1$ of $g^1$ in $\Gamma_1$ is nontrivial, the image of $h^1$ in some $G_{i1}$ 
is nontrivial.
Upon deleting some rows we may assume that $i=1$.
Thus $g^1$ has nontrivial image in $G_{11}$ and hence also in $G_{10}$.
By the above some $K_j$ contains $g^1$ and maps onto $H_1$.
Since $g^1\notin K_1$, we have $j\ge 2$, and then upon deleting some columns we may assume that $j=2$.
Let us pick any $a\in\Gamma_0\but K_2$.
Since $K_2$ maps onto $H_1$, some $b\in K_2$ has the same image in $H_1$ as $a$.
Then $g^2\bydef a-b$ lies in $\Gamma_0\but K_2$ and has trivial image in $G_{10}$.
Since the image $h^2$ of $g^2$ in $\Gamma_2$ is nontrivial,
the image of $h^2$ in $G_{i2}$ is nontrivial for some $i$.
We may assume that $i>1$, and then upon deleting some rows
we may assume that $i=2$.
Thus $g^2$ has nontrivial image in $G_{22}$ and hence also in $G_{20}$.
By the above some $K_j$ contains $g^2$ and maps onto $H_2$.
Since $g^2\notin K_2$, we have $j\ge 3$, and then upon deleting some columns 
we may assume that $j=3$.
By continuing in the same fashion, we will construct, upon passing to cofinal inverse and
direct sequences in both columns and rows, a sequence of elements $g^k\in K_{k+1}\but K_k$ 
such that the image of each $g^k$ in $G_{kk}$ is nontrivial, but the image of each $g^{k+1}$ in 
$G_{k0}$ is trivial.

Let $g^k_i$ denote the image of $g^k$ in $G_{i0}$.
Thus $g^k_i=0$ for all $k>i$.
Hence the elements $\gamma_i\bydef g^1_i+g^2_i+\dots+g^i_i\in G_{i0}$ form a thread 
$\gamma\bydef (\gamma_i)\in\Gamma_0$.%
\footnote{Let us note that $\gamma$ is the infinite sum $g^1+g^2+\dots$, in the sense that
it is the limit of the finite sums $g^1+\dots+g^i$ in the topology of the inverse limit 
of discrete groups.}
Since $g^k\in K_i$ for each $k<i$, the image of $g^k$ in $\Gamma_i$, and hence also in
$G_{ii}$ is trivial for $k<i$.
In other words, the image of $g^k_i$ in $G_{ii}$ is trivial for all $k<i$.
Hence the image of $\gamma_i$ in $G_{ii}$ equals that of $g^i_i$.
The latter is the same as the image of $g^i$ in $G_{ii}$, which is known to be nontrivial.
Thus $\gamma_i$ has nontrivial image in $G_{ii}$.
This is the same as the image of $\gamma$ in $G_{ii}$.
Then $\gamma$ has nontrivial image in $\Gamma_i$.
But then $\gamma\notin K_i$ for all $i$, contradicting $\Gamma_0=\bigcup_{i=1}^\infty K_i$.
\end{proof}

\begin{proof}[Proof of Theorem \ref{bounded-colim}(b)] We have $\Gamma_j=\lim_i G_{ij}$.
By the hypothesis $\lim^1_i G_{ij}=1$ for each $j$.
Since the $G_{ij}$ are countable, the tower $\dots\to G_{1j}\to G_{0j}$ satisfies the Mittag-Leffler condition for each $j$
(see \cite{M00}*{Theorem \ref{book:gray}(b)}).
In other words, for each $i$ and $j$ there exists a $k_{ij}$ such that
$H_{ij}\bydef \im(G_{k_{ij},j}\to G_{ij})$ equals $\im(G_{lj}\to G_{ij})$ for all $l\ge k_{ij}$.
It is easy to see that $H_{ij}=\im(H_{k_{ij},j}\to G_{ij})$, and it follows that
$H_{ij}=\im(\Gamma_j\to G_{ij})$.
Let $x_1,\dots,x_r$ be a set of generators of $G_{k_{ij},j}$, and let $g_1,\dots,g_r$
be their images in $G_{ij}$.
Then each $g_l\in H_{ij}$, so there exists a $\gamma_l\in\Gamma_j$ that maps onto $g_l$.
Since $\colim\Gamma_j=1$, there exists an $m_l$ such that $\gamma_l$ maps trivially to $\Gamma_{m_l}$.
Let $m_{ij}=\max(m_1,\dots,m_r)$.
Then each $\gamma_l$ maps trivially to $\Gamma_{m_{ij}}$.
Since $H_{ij}$ is generated by $g_1,\dots,g_r$, every $g\in H_{ij}$ is the image of some
$\gamma\in\Gamma_j$ that maps trivially to $\Gamma_{m_{ij}}$.
Using this, the proof of (a) now applies.
\end{proof}

\subsection{The case of lim$^1$: A warm-up}

\begin{proposition} \label{bounded-colim2}
Theorem \ref{bounded-colim}(c) holds if all vertical maps in the diagram are injective.
\end{proposition}

The proof works without significant changes in the non-abelian case.

\begin{proof} 
We have $\Gamma_j=\lim^1_i G_{ij}$.
Suppose that $\colim\Gamma_j=0$, but there exists a $j$ such that $\Gamma_j$
has a nontrivial image in each $\Gamma_k$.
Upon deleting some columns we may assume that $j=0$.
Upon deleting some more columns we may assume that there exist $g^k\in\Gamma_0$, 
$k=0,1,\dots$, such that the image $g^k_j$ of $g^k$ in $\Gamma_j$
is nontrivial for $j\le k$ and trivial for $j>k$.

Since the vertical maps are injective, for each $j$ we have a short exact sequence of
inverse sequences $0\to G_{ij}\to G_{0j}\to Q_{ij}\to 0$, where $Q_{ij}=G_{0j}/G_{ij}$.
This yields an exact sequence $0\to\lim_i G_{ij}\to G_{0j}\to\Lambda_j\to\Gamma_j\to 0$, where 
$\Lambda_j=\lim_i Q_{ij}$.
Thus each $g^k\in\Gamma_0$ is the image of some $\lambda^k\in\Lambda_0$.
Let $q^k_i$ be the image of $\lambda^k$ in $Q_{i0}$.
Each $q^k_k$ is the image of some $r_k\in G_{00}$.
By subtracting from $\lambda^k$ the image of $r_k$ in $\Lambda_0$ 
we may assume that $q^k_k=0$ and hence also $q^k_i=0$ for all $i\le k$.
Let $\lambda^k_j$ be the image of $\lambda^k$ in $\Lambda_j$, and let
$q^k_{ij}$ be the image of $\lambda^k_j$ (and of $q^k_i$) in $Q_{ij}$.
Then $q^k_{ij}=0$ for all $i\le k$ and all $j$.
On the other hand, $\lambda^k_k$ is nonzero, since it maps onto $g^k_k$, 
which is known to be nonzero.
Hence there exists an $i$ such that $q^k_{ik}\ne 0$.
By the above $i>k$, and then upon deleting some rows (and hence also the corresponding columns, 
to keep the correspondence between columns and rows) we may assume that $i=k+1$, that is, 
$q^k_{k+1,k}\ne 1$.
Then also $q^k_{k+1,j}\ne 1$ for all $j\le k$.

Given a subset $S\subset\N$, let $\mu^S\in\Lambda_0$ be the infinite sum $\sum_{k\in S}\lambda^k$, 
that is, the limit of the finite sums $\sum_{k\in S\cap[n]}\lambda^k$ in the topology of 
the inverse limit of discrete groups, where $[n]=\{0,\dots,n-1\}$.
In other words, $\mu^S$ is the thread consisting of the sums $\sigma_i^S\bydef \sum_{k\in S}q^k_i$, 
which are in fact finite sums $\sigma_i^S=\sum_{k\in S\cap[i]}q^k_i$.
Next let $\mu^S_j$ be the image of $\mu^S$ in $\Lambda_j$.
Then $\mu_j^S$ is the thread consisting of the sums
$\sigma_{ij}^S\bydef \sum_{k\in S}q^k_{ij}$, which are in fact finite sums
$\sigma_{ij}^S=\sum_{k\in S\cap[i]}q^k_{ij}$.
Let us note that $\mu^S_j\ne\mu^T_j$ as long as $S\but[j]\ne T\but[j]$.
Indeed, if $n$ is the smallest element of the symmetric difference $S\vartriangle T$ such that $n\ge j$, 
then the sums $\sigma_{n+1,j}^S=\sum_{k\in S\cap[n+1]}q^k_{n+1,j}$ and 
$\sigma_{n+1,j}^T=\sum_{k\in T\cap[n+1]}q^k_{n+1,j}$ 
differ by precisely one summand, $q^n_{n+1,j}$, which is nonzero by the above.

Thus if $S\vartriangle T$ is infinite, then $\mu^S_j\ne\mu^T_j$ for all $j$.
The relation ``$S\vartriangle T$ is finite'' is an equivalence relation on subsets of $\N$,
and every its equivalence class is countable.
By the axiom of choice there exists a set $U$ of subsets of $\N$ which contains precisely one
representative of each equivalence class.
Then $U$ is uncountable, and for any distinct $S,T\in U$ we have $\mu^S_j\ne\mu^T_j$ for all $j$.
Hence $\{\mu^S_j\mid S\in U\}$ is an uncountable subset of $\Lambda_j$ for each $j$.
Since $G_{0j}$ is countable, so is $V_j\bydef \{S\in U\mid\mu^S_j\in\im G_{0j}\}$ .
Then $V\bydef \bigcup_{j=0}^\infty V_j$ is countable as well.
Hence $U\but V$ is uncountable, and in particular nonempty.
Let us pick some $S\in U\but V$.
Then $\mu^S_j\notin\im G_{0j}$ for all $j$.
Hence the image $\gamma_j$ of $\mu^S_j$ in $\Gamma_j$ is nontrivial for all $j$.
Thus we have found an element $\gamma_0\in\Gamma_0$ which has nontrivial image (namely, $\gamma_j$)
in each $\Gamma_j$.
Hence $\colim\Gamma_j\ne 0$, which is a contradiction.
\end{proof}

\subsection{The case of lim$^1$: The main part}

\begin{proof}[Proof of Theorem \ref{bounded-colim}(d)] 
To simplify notation we will assume that each $G_{ij}$ is abelian, but the proof works without significant changes 
in the non-abelian case.
We have $\Gamma_j=\lim_i\Gamma_{ij}$, where $\Gamma_{ij}=\lim^1_l G_{ij}^{(l)}$.
Suppose that $\colim\Gamma_j=0$, but there exists a $j$ such that $\Gamma_j$ has a nontrivial image in each $\Gamma_k$.
Upon deleting some columns we may assume that $j=0$.
Upon deleting some more columns we may assume that there exist $g^k\in\Gamma_0$, $k=0,1,\dots$, 
such that the image $g^k_j$ of $g^k$ in $\Gamma_j$ is nonzero for $j\le k$ and zero for $j>k$.
Let $g^k_{ij}$ be the image of $g^k_j$ in $\Gamma_{ij}$ and let $J_{ij}=\{k\in\N\mid g^k_{ij}\ne 0\}$.
Since $g^k_j=0$ for $k<j$, we have $J_{ij}\subset\N\but[j]$, where $[j]=\{0,\dots,j-1\}$.
Also, since $g^k_{ij}$ is the image of $g^k_{lj}$ for $i\le l$, we have $J_{ij}\subset J_{lj}$ if $i\le l$.

Suppose that there exist an $i_0$ and a $j$ such that $J_{i_0j}$ is infinite.
Upon deleting some columns we may assume that $j=0$.
Since $g^0_0\ne 0$, there exists an $i$ such that $g^0_{i0}\ne 0$.
Since $J_{ij}\subset J_{lj}$ if $i\le l$, we may assume that $i_0\ge i$.
Thus $J_{i_00}$ is infinite and contains $0$.
Upon deleting the $k$th column (and hence forgetting $g^k$) for all $k\notin J_{i_00}$ 
we may assume that $J_{i_00}=\N$ (that is, $g^k_{i_00}\ne 0$ for all $k$).
Next, suppose that (after the said columns have been deleted) there exist an $i_1$ and $j>0$ 
such that $J_{i_1j}$ is infinite.
Upon deleting some columns we may assume that $j=1$.
Since $g^1_1\ne 0$, there exists an $i$ such that $g^1_{i1}\ne 0$.
Since $J_{ij}\subset J_{lj}$ if $i\le l$, we may assume that $i_1\ge\max(i_0+1,i)$.
Thus $J_{i_11}$ is infinite and contains $1$.
(As noted above, $J_{i_11}$ lies in $\N\but[1]$, i.e.\ does not contain $0$.)
Upon deleting the $k$th column (and hence forgetting $g^k$) for all $k\in(\N\but[1])\but J_{i_11}$ 
we may assume that $J_{i_11}=\N\but[1]$ (that is, $g^k_{i_11}\ne 0$ for all $k\ge 1$).
By proceeding in the same fashion we will eventually construct (upon deleting some columns) 
either an infinite sequence $i_0,i_1,\dots$ such that each $J_{i_k,k}=\N\but[k]$ or
a finite sequence $i_0,\dots,i_n$ such that $J_{ij}$ is finite for all $j>n$ and all $i$.
In the former case we may assume, upon deleting some rows, that each $i_k=k$; and in 
the latter case we may assume, upon deleting some columns, that $J_{ij}$ is finite for all $i$ and $j$.

To summarize, we have achieved (upon deleting some columns and possibly rows) that either of 
the following holds:
\begin{itemize}
\item Case I. $J_{ij}$ is finite for all $i$ and $j$;
\item Case II. $J_{jj}=\N\but[j]$ for all $j$.
\end{itemize}
The proof in these two cases will modeled on the proofs of (a) and (b), respectively.

{\it Case I.}
Since $g^0_0\ne 0$, there exists an $i$ such that $g^0_{i0}\ne 0$.
Upon deleting some rows we may assume that $g^0_{00}\ne 0$.
On the other hand, by deleting the $k$th column (and hence forgetting $g^k$) for all $k\in J_{00}\but\{0\}$
we will get that $g^k_{00}=0$ for all $k\ge 1$.
Thus $J_{00}=\{0\}$.
Next, since $g^0_1=0$, we have $g^0_{i1}=0$ for all $i$; but since $g^1_1\ne 0$, there exists an $i$ 
such that $g^1_{i1}\ne 0$.
We may assume that $i\ge 1$, and then upon deleting some rows we may assume that $g^1_{11}\ne 0$.
On the other hand, by deleting the $k$th column (and hence forgetting $g^k$) for all $k\in J_{11}\but\{1\}$
we will get that $g^k_{11}=0$ for all $k\ge 2$.
Thus $J_{11}=\{1\}$.
By proceeding in this fashion we will eventually obtain (upon deleting some columns and rows) that 
$J_{jj}=\{j\}$ for all $j$.

Now we have $g^j_{jj}\ne 0$ for all $j$ and $g^k_{jj}=0$ for all $k>j$.
Let $\gamma\in\Gamma_0$ be the infinite sum $\sum_{k\in\N}g^k$, that is, the limit of the finite sums 
$\sum_{k\in [n]}g^k$ in the topology of the inverse limit of discrete groups.
Thus $\gamma$ is the thread consisting of the sums $\sigma_i\bydef \sum_{k\in\N}g^k_{i0}$, 
which are in fact finite sums $\sigma_i=\sum_{k\in J_{i0}}g^k_{i0}$.
The image $\gamma_j$ of $\gamma$ in $\Gamma_j$ is the thread consisting of the sums 
$\sigma_{ij}\bydef \sum_{k\in\N}g^k_{ij}$, which are in fact finite sums 
$\sigma_{ij}=\sum_{k\in J_{ij}}g^k_{ij}$.
Since $J_{jj}=\{j\}$ for each $j$, we have $\sigma_{jj}=g^j_{jj}\ne 0$ for each $j$.
Hence $\gamma_j\ne 0$ for each $j$.
Thus we have found an element $\gamma\in\Gamma_0$ which has nonzero image (namely, $\gamma_j$)
in each $\Gamma_j$.
Hence $\colim\Gamma_j\ne 0$, which is a contradiction.

{\it Case II.}
Let $\Delta_j=\lim^1_i G_{ij}$.
Since the map $\Delta_0\to\Gamma_0$ is surjective by the Mittag-Leffler short exact sequence
(see \cite{M00}*{Theorem \ref{book:mles2}}), some $h^k\in\Delta_j$ maps onto $g^k$.

Let us embed each vertical inverse sequence in an inverse sequence of surjections.
For each $i$ and $j$ let $P_{ij}=G_{ij}\oplus G_{i-1,j}\oplus\dots\oplus G_{0j}$ and let
$f\:G_{ij}\to P_{ij}$ be defined by $f(g)=\big(g,\,p^{ij}_{i-1,j}(g),\dots,p^{ij}_{0j}(g)\big)$.
Clearly $f$ is injective.
Let $Q_{ij}=P_{ij}/f(G_{ij})$.
It is easy to see that the following diagram commutes for each $i$ and $j$:
\[\begin{tikzcd}[row sep=1.3em,column sep=0.1em]
& & {Q_{i+1,\,j}} \ar[rrrr] \ar[ddd] & & & & {Q_{i+1,\,j+1}} \ar[ddd] \\
& {P_{i+1,\,j}} \ar[rrrr,crossing over] \ar[ddd] \ar[ru,twoheadrightarrow] & & & & {P_{i+1,\,j+1}} \ar[ru,twoheadrightarrow] & \\
{G_{i+1,\,j}} \ar[ddd] \ar[rrrr,crossing over] \ar[ru,rightarrowtail] & & & & {G_{i+1,\,j+1}} \ar[ru,rightarrowtail] & & \\
& & Q_{ij} \ar[rrrr] & & & & {Q_{i+1,\,j+1}} \\
& P_{ij} \ar[rrrr] \ar[ru,twoheadrightarrow] & & & & {P_{i,\,j+1}} \ar[ru,twoheadrightarrow] \ar[from=uuu,crossing over] & \\
G_{ij} \ar[rrrr] \ar[ru,rightarrowtail] & & & & {G_{i,\,j+1}} \ar[ru,rightarrowtail] \ar[from=uuu,crossing over] & &                         
\end{tikzcd}\]
where each vertical map $P_{i+1,j}\to P_{ij}$ is the projection along the first summand; each horizontal map $P_{ij}\to P_{i,\,j+1}$ is 
given by the horizontal maps $G_{lj}\to G_{l,j+1}$ for $l=i,\dots,0$; and each vertical map $Q_{i+1,\,j}\to Q_{ij}$ and each horizontal map 
$Q_{ij}\to Q_{i,\,j+1}$ are yielded respectively by the commutative squares
\[\begin{CD}
G_{i+1,j}@>>>P_{i+1,j}\\
@VVV@VVV\\
G_{ij}@>>>P_{ij}
\end{CD}
\qquad\text{ and }\qquad
\begin{CD}
G_{ij}@>>>P_{ij}\\
@VVV@VVV\\
G_{i,j+1}@>>>P_{i,j+1}.\!
\end{CD}\]

This yields for each $j$ an exact sequence $0\to\lim_i G_{ij}\to\lim_i P_{ij}\to\Lambda_j\to\Delta_j\to 0$, where
$\Lambda_j=\lim_i Q_{ij}$.
Thus each $h^k\in\Delta_0$ is the image of some $\lambda^k\in\Lambda_0$.
Let $q^{kl}$ be the image of $\lambda^k$ in $Q_{l0}$.
Each $q^{kk}\in Q_{k0}$ is the image of some $r_k\in P_{k0}$.
Since the maps $P_{l+1,0}\to P_{l0}$ are surjective, $r_k$ is the image of some $\rho_k\in\lim_l P_{l0}$.
By subtracting from $\lambda^k$ the image of $\rho_k$ in $\Lambda_0$ we may assume that $q^{kk}=0$ and hence also 
$q^{kl}=0$ for all $l\le k$.

Next, for each $i$ and each $j$ we have a short exact sequence of
inverse sequences $0\to G_{ij}^{(l)}\to P_{ij}\to Q_{ij}^l\to 0$, where $Q_{ij}^l=P_{ij}/G_{ij}^{(l)}$.
This yields an exact sequence $0\to\lim_l G_{ij}^{(l)}\to P_{ij}\to\Lambda_{ij}\to\Gamma_{ij}\to 0$, where 
$\Lambda_{ij}=\lim_l Q_{ij}^l$.
The commutative diagram
\[\begin{CD}
0@>>>G_{lj}@>>>P_{lj}@>>>Q_{lj}@>>>0\\
@.@VVV@VVV@VVV@.\\
0@>>>G_{ij}^{(l)}@>>>P_{ij}@>>>Q_{ij}^l@>>>0,\!
\end{CD}\]
where the vertical arrow on the right is yielded by the commutative square on the left, yields 
a commutative diagram
\[\begin{CD}
0@>>>\lim_l G_{lj}@>>>\lim_l P_{lj}@>>>\Lambda_j@>>>\Delta_j@>>>0\\
@.@VVV@VVV@VVV@VVV@.\\
0@>>>\lim_l G_{ij}^{(l)}@>>>P_{ij}@>>>\Lambda_{ij}@>>>\Gamma_{ij}@>>>0.\!
\end{CD}\]

Let $\lambda^k_j$ be the image of $\lambda^k$ in $\Lambda_j$, and let
$\lambda^k_{ij}$ be the image of $\lambda^k_j$ in $\Lambda_{ij}$.
Let $q^{kl}_j$ be the image of $\lambda^k_j$ (and of $q^{kl}$) in $Q_{lj}$.
Then $q^{kl}_j=0$ for all $l\le k$ and all $j$.
Let $q^{kl}_{ij}$ be the image of $\lambda^k_{ij}$ (and of $q^{kl}_j$) in $Q_{ij}^l$.
Then $q^{kl}_{ij}=0$ for all $l\le k$ and all $i$ and $j$ (where $i\ge l$ holds automatically
by the definition of $Q_{ij}^l$).
In particular, $q^{kl}_{jj}=0$ for $l\le k$.
On the other hand, we know that $g^k_{jj}\in\Gamma_{jj}$ is nonzero for each $j\le k$, since $J_{jj}=\N\but[j]$.
Then $\lambda^k_{jj}\in\Lambda_{jj}$ is nonzero for all $j\le k$, since it maps onto $g^k_{jj}$.
Hence for each $j=0,\dots,k$ there exists an $l_{kj}$ such that $q^{kl_j}_{jj}\ne 0$.
Let $l_k=\max(l_{k0},\dots,l_{kk})$.
Then $q^{kl_k}_{jj}\ne 0$, and since $q^{kl}_{jj}=0$ for $l\le k$, we have $l_k>k$.
Upon deleting the rows from $k+1$ to $l_k-1$, where $k$ runs over $0,l_0,l_{l_0},\dots$ (and hence also 
the corresponding columns, to keep the correspondence between columns and rows) we may assume 
that $l_k=k+1$ for each $k$.
Thus $q^{k,k+1}_{jj}\ne 0$ for all $j\le k$.
Hence also $q^{kl}_{jj}\ne 0$ for all $j\le k$ and all $l>k$.

Given a subset $S\subset\N$, let $\mu^S\in\Lambda_0$ be the infinite sum $\sum_{k\in S}\lambda^k$, 
that is, the limit of the finite sums $\sum_{k\in S\cap[n]}\lambda^k$ in the topology of 
the inverse limit of discrete groups.
In other words, $\mu^S$ is the thread consisting of the sums $\sigma^{Sl}\bydef \sum_{k\in S}q^{kl}$, 
which are in fact finite sums $\sigma^{Sl}=\sum_{k\in S\cap[l]}q^{kl}$.
Next let $\mu^S_j$ be the image of $\mu^S$ in $\Lambda_j$ and let $\mu^S_{ij}$ be the image of $\mu^S_j$
in $\Lambda_{ij}$.
Then $\mu^S_{ij}$ is the thread consisting of the sums $\sigma^{Sl}_{ij}\bydef \sum_{k\in S}q^{kl}_{ij}$
(where $l\ge i$ for $Q_{ij}^l$ to be defined), which are in fact finite sums 
$\sigma^{Sl}_{ij}=\sum_{k\in S\cap[l]}q^{kl}_{ij}$.
Let us note that $\mu^S_{jj}\ne\mu^T_{jj}$ as long as $S\but[j]\ne T\but[j]$.
Indeed, if $n$ is the smallest element of the symmetric difference $S\vartriangle T$ such that $n\ge j$, 
then the sums $\sigma^{S,n+1}_{jj}=\sum_{k\in S\cap[n+1]}q^{k,n+1}_{jj}$ and 
$\sigma^{T,n+1}_{jj}=\sum_{k\in T\cap[n+1]}q^{k,n+1}_{jj}$ 
differ by precisely one term, $q^{n,n+1}_{jj}$, which is nonzero by the above.

Thus if $S\vartriangle T$ is infinite, then $\mu^S_{jj}\ne\mu^T_{jj}$ for all $j$.
The relation ``$S\vartriangle T$ is finite'' is an equivalence relation on subsets of $\N$,
and every its equivalence class is countable.
By the axiom of choice there exists a set $U$ of subsets of $\N$ which contains precisely one
representative of each equivalence class.
Then $U$ is uncountable, and for any distinct $S,T\in U$ we have $\mu^S_{jj}\ne\mu^T_{jj}$ for all $j$.
Hence $\{\mu^S_{jj}\mid S\in U\}$ is an uncountable subset of $\Lambda_{jj}$ for each $j$.
Since $P_{jj}$ is countable, $V_j\bydef \{S\in U\mid\mu^S_{jj}\in\im P_{jj}\}$ is also countable.
Then $V\bydef \bigcup_{j=0}^\infty V_j$ is countable as well.
Hence $U\but V$ is uncountable, and in particular nonempty.
Let us pick some $S\in U\but V$.
Then $\mu^S_{jj}\notin\im P_{jj}$ for all $j$.
Hence the image $\gamma_{jj}$ of $\mu^S_{jj}$ in $\Gamma_{jj}$ is nonzero for all $j$.
Then also the image $\gamma_j$ of $\mu^S_j$ under the composition $\Lambda_j\to\lim_i\Lambda_{ij}\to\lim_i\Gamma_{ij}=\Gamma_j$ 
is nonzero for all $j$.
Thus we have found an element $\gamma_0\in\Gamma_0$ which has nonzero image (namely, $\gamma_j$)
in each $\Gamma_j$.
Hence $\colim\Gamma_j\ne 0$, which is a contradiction.
\end{proof}

\subsection{The case of lim$^1$: The final step}

\begin{proposition} \label{bounded-colim3}
Theorem \ref{bounded-colim}(c) holds if each map $\Gamma_i\to\Gamma_{i+1}$ is surjective.
\end{proposition}

This is proved using Theorem \ref{bounded-colim}(d) and will be used in the proof of Theorem \ref{bounded-colim}(c).
The proof works without significant changes in the non-abelian case.

\begin{proof}
We have $\Gamma_j=\lim^1_i G_{ij}$.
By the hypothesis the map $\Gamma_0\to\Gamma_j$ is surjective for each $j$.
Suppose that $\colim\Gamma_j=0$, but there exists a $k$ such that $\Gamma_k$ has a nontrivial image in each $\Gamma_j$.
Upon deleting some columns we may assume that $k=0$.
Then in particular each $\Gamma_j\ne 0$.
Then for each $j$ the inverse sequence $\dots\to G_{1j}\to G_{0j}$ does not satisfy the Mittag-Leffler condition, 
i.e.\ if $G_{kj}^{(i)}$ denotes the image of $G_{ij}$ in $G_{kj}$,
then for each $j$ there exists a $k_j$ such that the groups $G_{k_jj}^{(i)}$ do not stabilize as $i\to\infty$.
Since each $G_{k_jj}$ is countable, $\lim^1_i G_{k_jj}^{(i)}\ne 0$.
Each vertical map $G_{k+1,j}^{(i)}\to G_{kj}^{(i)}$ is surjective, and 
hence so is the map $\lim^1_i G_{k+1,j}^{(i)}\to\lim^1_i G_{kj}^{(i)}$.
Therefore $\Delta_j\bydef \lim_k\lim^1_i G_{kj}^{(i)}$ is nonzero for each $j$.
Let $\delta_j$ be a nonzero element of $\Delta_j$.
Since the map $\Gamma_j\to\Delta_j$ is surjective by the Mittag-Leffler short exact sequence 
(see \cite{M00}*{Theorem \ref{book:mles2}}), some $\gamma_j\in\Gamma_j$ maps onto $\delta_j$.
Since the map $\Gamma_0\to\Gamma_j$ is surjective, some $\gamma\in\Gamma_0$ maps onto $\gamma_j$.
Let $\delta$ be the image of $\gamma$ in $\Delta_0$.
Since $\delta$ maps onto $\delta_j$, and $j$ was arbitrary, we obtain that $\Delta_0$ has nontrivial image 
in each $\Delta_j$.
On the other hand, since $\colim\Gamma_j=0$, $\gamma_j$ maps trivially to some $\Gamma_k$.
Then $\delta_j$ maps trivially to $\Delta_k$.
Since $\delta_j$ was an arbitrary nonzero element of $\Delta_j$, where $j$ was arbitrary, we conclude that 
$\colim\Delta_j=0$.
Hence by Theorem \ref{bounded-colim}(d) $\Delta_0$ maps trivially to some $\Delta_k$, which is a contradiction.
\end{proof}

\begin{proof}[Proof of Theorem \ref{bounded-colim}(c)]
We have $\Gamma_j=\lim^1_i G_{ij}$.
Suppose that $\colim\Gamma_j=0$, but there exists a $k$ such that $\Gamma_k$ has a nonzero image in each $\Gamma_j$.
Upon omitting some columns in the diagram we may assume that $k=0$.

By Lemma \ref{level-factorization}(b) the level map $G_{i0}\to G_{i1}$ factors through
a tower $H_{i1}$ such that the map $\lim^1 G_{i0}\to\lim^1 H_{i1}$ is surjective and
the map $\lim^1 H_{i1}\to\lim^1 G_{i1}$ is injective.
Again by Lemma \ref{level-factorization}(b) the composite level map $H_{i1}\to G_{i1}\to G_{i2}$ 
factors through a tower $H_{i2}$ such that the map $\lim^1 H_{i1}\to\lim^1 H_{i2}$ is surjective and
the map $\lim^1 H_{i2}\to\lim^1 G_{i2}$ is injective.
By continuing in the same fashion we obtain a commutative diagram of level maps of towers
\[\begin{tikzcd}
G_{i0} \drar \rar & H_{i1} \rar \dar & H_{i2} \rar \dar & H_{i3} \rar \dar & \dots \\
& G_{i1} \rar & G_{i2} \rar & G_{i3} \rar & \dots
\end{tikzcd}\]
where the horizontal maps in the upper row induce surjections on $\lim^1$ and the vertical maps induce 
injections on $\lim^1$.

Let $\Delta_i=\lim^1_i H_{ij}$ for $i\ge 1$.
Since $\colim\Gamma_j=0$, it is easy to see that $\colim\Delta_j=0$.
(Indeed, since each element of $\Gamma_i$ maps to zero in some $\Gamma_j$, and the map $\Delta_j\to\Gamma_j$ 
is injective, each element of $\Delta_i$ maps to zero in $\Delta_j$.)
Then by Proposition \ref{bounded-colim2} we get that each $\Delta_i$ maps trivially to some $\Delta_j$.
Then $\Gamma_0$ maps trivially to some $\Gamma_j$, contradicting our hypothesis.
\end{proof}

\section{Applications to topology}

In this section we prove Theorems \ref{ind-colimit} and \ref{ind-isomorphism2}.

We will use the following terminology.
Given inverse sequences of metrizable spaces $\PP=\big(\dots\xr{p_1} P_1\xr{p_0} P_0\big)$ and $\QQ=\big(\dots\xr{q_1} Q_1\xr{q_0} Q_0)$,
by an {\it h-inv-map} $f\:\PP\to\QQ$ we mean a sequence $n_0,n_1,\dots$ along with a collection of maps $f_i\:P_{n_i}\to Q_i$
such that for each $i$ there exists a $j\ge n_{i+1}$ such that the diagram
\[\begin{tikzcd}[row sep=1em]
& P_{n_{i+1}} \ar[r, "f_{i+1}"] &  Q_{i+1} \ar[dd, "q_i"] \\
P_j \ar[ru, "p^j_{n_{i+1}}"] \ar[rd, "p^j_{n_i}"']  & & \\
& P_{n_i} \ar[r, "f_i"] &  Q_i                              
\end{tikzcd}\]
commutes up to homotopy.
(Thus an h-inv-map is a representative of an inv-$\Ho$-morphism.)
We call $f$ an {\it h-level map} if each $n_i=i$ and the above diagram exists already for $j=i+1$.

\subsection{Vanishing: the absolute case}

Theorem \ref{ind-colimit} is a consequence of the following result along with Lemma \ref{Cech}.

\begin{theorem} \label{ind-colimit0}
Let $X$ be a local compactum and let $X_0\subset X_1\subset\dots$ be compact subsets of 
$X$ such that $\bigcup_i X_i=X$ and each $X_i\subset\Int X_{i+1}$.

(a) $H_n(X)=0$ if and only if for each $j$ there exists a $k\ge j$ such that the inclusion induced map 
$H_n(X_j)\to H_n(X_k)$ is trivial.

(b) Let $x\in X_0$.
Suppose that $X$ is locally connected and either $n\le 1$ or $\breve\pi_1(X,x)=1$.
Then $\pi_n(X,x)$ is trivial if and only if for each $j$ there exists a $k\ge j$ such that the inclusion induced map 
$\pi_n(X_j,x)\to\pi_n(X_k,x)$ is trivial.
\end{theorem}

\begin{proof}[Proof. (a)] The ``if'' assertion is trivial.

Let us prove the ``only if'' assertion. 
Let us represent $X$ as the limit of a scalable inverse sequence of locally compact polyhedra
$\dots\xr{p_1} |K_1|\xr{p_0} |K_0|$ (see \cite{M00}*{Theorem \ref{book:isbell}}).
Let $P_{ij}$ be the union of all simplexes of $K_i$ that meet the image of $X_j$.
Then $P_{ij}$ is a compact polyhedron, $p_i(P_{i+1,j})\subset P_{ij}$, and each $X_j$ is homeomorphic 
to $\lim\big(\dots\to P_{1j}\to P_{0j}\big)$.
The Milnor short exact sequences yield for every $j\le k\le l$ a commutative diagram with exact rows
\[\begin{tikzcd}
0\rar &\lim\limits_i\!^1\, H_{n+1}(P_{ij})\rar\dar &H_n(X_j)\rar\dar &\lim\limits_i H_n(P_{ij})\rar\dar &0\\
0\rar &\lim\limits_i\!^1\, H_{n+1}(P_{ik})\rar\dar &H_n(X_k)\rar\dar &\lim\limits_i H_n(P_{ik})\rar\dar &0\\
0\rar &\lim\limits_i\!^1\, H_{n+1}(P_{il})\rar &H_n(X_l)\rar &\lim\limits_i H_n(P_{il})\rar &0.\!
\end{tikzcd}\]
Since $\colim H_n(X_j)=H_n(X)=0$, using the upper half of this diagram we easily get that
$\colim_j\lim_i H_n(P_{ij})=0$ and $\colim_j\lim^1_i H_{n+1}(P_{ij})=0$.
Then by Theorem \ref{bounded-colim}(a) for each $j$ there exists a $k\ge j$ such that 
the map $\lim_i H_n(P_{ij})\to\lim_i H_n(P_{ik})$ is trivial.
Moreover, by Theorem \ref{bounded-colim}(c) there further exists an $l\ge k$ such that 
the map $\lim^1_i H_{n+1}(P_{ik})\to\lim^1_i H_{n+1}(P_{il})$ is also trivial.
Then it is easy to see from the above diagram that the map $H_n(X_j)\to H_n(X_l)$ is trivial.
\end{proof}

\begin{proof}[(b)] Since $X$ is locally connected, each inclusion map $X_j\to X_{j+1}$ factors through a locally connected compactum $Y_j$ 
\cite{M1}*{proof of Theorem 6.11(a)}.
Let us note that $Y_j$ contains $X_j$ and in particular the basepoint $x$.
To prove that each $\pi_n(X_j,x)$ maps trivially to $\pi_n(X_k,x)$ for some $k$ it suffices to show that
$\pi_n(Y_{j+1},x)$ maps trivially to $\pi_n(Y_k,x)$, for then the composition 
$\pi_n(X_j,x)\to\pi_n(Y_{j+1},x)\to\pi_n(Y_k,x)\to\pi_n(X_k,x)$ is also trivial.

The case $n=0$ is easy.
Since $Y_j$ is locally connected, by Lemma \ref{continua} $\pi_0(Y_j,x)$ is finite for each $j$.
Then by Proposition \ref{obvious} each $\pi_0(Y_j,x)$ maps trivially to some $\pi_0(Y_k,x)$.

Let us represent $Y_j$ as the limit of an inverse sequence of compact polyhedra $\QQ_j=\big(\dots\to Q_{1j}\to Q_{0j}=pt\big)$.
Let $q_{ij}$ be the image of $x$ in $Q_{ij}$.
Each composition $Y_j\to X_{j+1}\to Y_{j+1}$ extends to a map $Q_{[0,\infty],\,j}\to Q_{[0,\infty],\,j+1}$ of the extended mapping
telescopes which comes from an h-inv-map $\phi_j\:\QQ_j\to\QQ_{j+1}$ and sends the extended mapping telescope $q_{[0,\infty],\,j}$ 
of the singletons $\dots\to\{q_{1j}\}\to\{q_{0j}\}$ into $q_{[0,\infty],\,j+1}$ (see \cite{M00}*{Theorem \ref{book:strongshape}(a)
and Remark \ref{book:telescope0}} or \cite{M1}*{proof of Lemma 2.1(a)}).
By using Lemma \ref{straightening} we may assume that each $\phi_j$ is an h-level map.

The Quigley short exact sequences (see \cite{M1}*{Theorem 3.1(b)}) yield for every $j\le k\le l$ a commutative diagram with exact rows
\[\begin{tikzcd}
1\rar &\lim\limits_i\!^1\,\pi_{n+1}(Q_{ij},q_{ij})\rar\dar &\pi_n(Y_j,x)\rar\dar &\lim\limits_i\pi_n(Q_{ij},q_{ij})\rar\dar &1\\
1\rar &\lim\limits_i\!^1\,\pi_{n+1}(Q_{ik},q_{ik})\rar\dar &\pi_n(Y_k,x)\rar\dar &\lim\limits_i\pi_n(Q_{ik},q_{ik})\rar\dar &1\\
1\rar &\lim\limits_i\!^1\,\pi_{n+1}(Q_{il},q_{il})\rar &\pi_n(Y_l,x)\rar &\lim\limits_i\pi_n(Q_{il},q_{il})\rar &1,\!
\end{tikzcd}\]
which consists of abelian groups for $n\ge 2$, groups for $n=1$ and pointed sets for $n=0$.
Since $\colim\pi_n(Y_j,x)\simeq\colim\pi_n(X_j,x)=\pi_n(X,x)=*$, using the upper half of this diagram we easily get that
$\colim_j\lim_i\pi_n(Q_{ij},q_{ij})=*$ and $\colim_j\lim^1_i\pi_{n+1}(Q_{ij},q_{ij})=*$.

Let $G_{ij}=\pi_1(Q_{ij},q_{ij})$.
Since $Y_j$ is locally connected, by Lemma \ref{continua} the map $\pi_0(Y_j,x)\to\lim_i\pi_0(Q_{ij},q_{ij})$ is bijective and hence
$\lim^1_i G_{ij}$ is trivial for each $j$.
On the other hand, having already considered the case $n=0$, we may assume that $n\ge 1$.
Then by the hypothesis either $n=1$ or $\breve\pi_1(X,x)=1$, and in both cases we get that $\colim_j\lim_i G_{ij}=1$.
Then by Theorem \ref{bounded-colim}(b) for each $j$ there exists 
a $k\ge j$ such that the map $\lim_i G_{ij}\to\lim_i G_{ik}$ is trivial.
Moreover, by Theorem \ref{bounded-colim}(c) there further exists an $l\ge k$ such that 
the map $\lim^1_i\pi_2(Q_{ik},q_{ik})\to\lim^1_i\pi_2(Q_{il},q_{il})$ is also trivial.
Then it is easy to see from the above diagram that the map $\pi_1(Y_j,x)\to\pi_1(Y_l,x)$ is trivial.
This completes the proof of the case $n=1$.

In the case $n\ge 2$ we have already shown that for each $j$ there exists a $k>j$ such that the map 
$\lim_i G_{ij}\to\lim_i G_{ik}$ is trivial.
By omitting some of the $X_j$ we may assume that $k=j+1$.
Since $\lim^1_i G_{ij}$ is trivial, the tower $\dots\to G_{1j}\to G_{0j}$ satisfies 
the Mittag-Leffler condition (see \cite{M00}*{Theorem \ref{book:gray}(b)}).
Then by Lemma \ref{level-factorization3} for each $j$ the map of towers of groups $G_{ij}\to G_{i,\,j+1}$ factors through 
a tower of groups $B_{ij}$ such that both $\lim_i B_{ij}$ and $\lim^1_i B_{ij}$ are trivial.
Moreover each $B_{ij}$ is a quotient of $G_{ij}$ and hence is countable.
Then the tower $\dots\to B_{1j}\to B_{0j}$ satisfies the Mittag-Leffler condition, and since it also has trivial inverse limit,
it is trivial as a pro-group, that is, for each $i$ there exists an $l\ge i$ such that the map $B_{lj}\to B_{ij}$ is trivial 
(see \cite{M1}*{Lemma 3.4(a)}).
Therefore the composition $G_{lj}\to B_{lj}\to B_{ij}\to G_{i,\,j+1}$ is trivial.
Using Lemma \ref{straightening} we may assume that already the maps $G_{ij}\to G_{i,\,j+1}$ are trivial.

Let $\bar Q_{ij}$ be the connected component of $Q_{ij}$ containing $q_{ij}$, and let $\bar Q_{ij}^{(1)}$ be
the $1$-skeleton of some triangulation of $\bar Q_{ij}$ which has $q_{ij}$ as a vertex.
Let $\hat Q_{ij}=Q_{ij}\cup C(\bar Q_{ij}^{(1)})$, where $CP$ denotes the cone over the polyhedron $P$.
Then the map $(Q_{ij},q_{ij})\to (Q_{i,\,j+1},q_{i,\,j+1})$ factors through $(\hat Q_{ij},q_{ij})$.
Also each bonding map $Q_{i+1,j}\to Q_{ij}$ extends to a map $\hat Q_{i+1,j}\to\hat Q_{ij}$.

Since $\pi_1(\hat Q_{ij},q_{ij})=1$, by Serre's theorem on classes of abelian groups (see \cite{Sp}*{9.6.16}) 
the abelian group $\pi_n(\hat Q_{ij},q_{ij})$ is finitely generated.
Then by Theorem \ref{bounded-colim}(a) for each $j$ there exists a $k>j$ such that the middle map in
the composition $\lim_i\pi_n(Q_{ij},q_{ij})\to\lim_i\pi_n(\hat Q_{ij},q_{ij})\to\lim_i\pi_n(\hat Q_{i,\,k-1},q_{i,\,k-1})
\to\lim_i\pi_n(Q_{ik},q_{ik})$ is trivial.
Moreover, by Theorem \ref{bounded-colim}(c) there further exists an $l>k$ such that the middle map in
the composition $\lim^1_i\pi_{n+1}(Q_{ik},q_{ik})\to\lim^1_i\pi_{n+1}(\hat Q_{ik},q_{ik})\to\lim^1_i\pi_{n+1}(\hat Q_{i,\,l-1},q_{i,\,l-1})
\to\lim^1_i\pi_{n+1}(Q_{il},q_{il})$ is trivial.
Then it is easy to see from the above diagram that the map $\pi_n(Y_j,x)\to\pi_n(Y_l,x)$ is trivial.
\end{proof}

\subsection{Vanishing: the relative case}

Next we prove a relative version of Theorem \ref{ind-colimit0}.

\begin{theorem} \label{ind-colimit2}
Let $X$ be a local compactum and let $X_0\subset X_1\subset\dots$ be compact subsets of 
$X$ such that $\bigcup_i X_i=X$ and each $X_i\subset\Int X_{i+1}$.
Let $A$ be a closed subset of $X$ and let $A_i=X_i\cap X$.

(a) $H_n(X,A)=0$ if and only if for each $j$ there exists a $k\ge j$ such that the inclusion induced map 
$H_n(X_j,A_j)\to H_n(X_k,A_k)$ is trivial.

(b) Let $x\in A_0$ and suppose that $X$ and $A$ are locally connected and $\breve\pi_1(X,x)=\breve\pi_1(A,x)=1$.
Then $\pi_n(X,A,x)$ is trivial if and only if for each $j$ there exists a $k\ge j$ such that the inclusion induced map 
$\pi_n(X_j,A_j,x)\to\pi_n(X_k,A_k,x)$ is trivial.
\end{theorem}

\begin{proof}[Proof. (a)] Similarly to the proof of Theorem \ref{ind-colimit}(a).
\end{proof}

\begin{proof}[(b)] Since $X$ and $A$ are locally connected, each inclusion map $(X_j,A_j)\to (X_{j+1},A_{j+1})$ factors through 
a pair of locally connected compacta $(Y_j,B_j)$ similarly to \cite{M1}*{proof of Theorem 6.11(a)}.
Let us note that $B_j$ contains $A_j$ and in particular the basepoint $x$.
Since $Y_j$ is locally connected, by Lemma \ref{continua} $\pi_0(Y_j,x)$ is finite.
Then so is its quotient $\pi_0(Y_j,B_j,x)$. 
Now the proof of the case $n=0$ is similar to that in Theorem \ref{ind-colimit}(b).

Let us represent $(Y_j,B_j)$ as the limit of an inverse sequence compact polyhedral pairs $\dots\to (Q_{1j},R_{1j})\to 
(Q_{0j},R_{0j})=(pt,pt)$.
Let $q_{ij}$ be the image of $x$ in $R_{ij}$.
Similarly to the proof of Theorem \ref{ind-colimit}(b) each composition $(Y_j,B_j)\to (X_{j+1},A_{j+1})\to (Y_{j+1},B_{j+1})$ extends to 
a map $(Q_{[0,\infty],\,j},R_{[0,\infty],\,j})\to (Q_{[0,\infty],\,j+1},R_{[0,\infty],\,j})$ of the pairs of extended mapping
telescopes which comes from an h-level map $(Q_{ij},R_{ij})\to(Q_{i,\,j+1},R_{i,\,j+1})$ and sends the base ray
$q_{[0,\infty],\,j}$ into the base ray $q_{[0,\infty],\,j+1}$.

Since the $Y_j$ and $B_j$ are locally connected, and $\breve\pi_1(X,x)$ and $\breve\pi_1(A,x)$ are trivial,
similarly to the proof of Theorem \ref{ind-colimit}(b) we may assume that the maps 
$\pi_1(Q_{ij},q_{ij})\to\pi_1(Q_{i,\,j+1},q_{i,\,j+1})$ and
$\pi_1(R_{ij},q_{ij})\to\pi_1(R_{i,\,j+1},q_{i,\,j+1})$ are trivial.
Let $\bar Q_{ij}$ and $\bar R_{ij}$ be the connected components of $Q_{ij}$ and $R_{ij}$ containing $q_{ij}$,
and let $(K_{ij},L_{ij})$ be a triangulation of $(\bar Q_{ij},\bar R_{ij})$ which has $q_{ij}$ as a vertex.
Let $\bar Q_{ij}^{(1)}$ and $\bar R_{ij}^{(1)}$ be the $1$-skeleta of $K_{ij}$ and $L_{ij}$.
Let $\hat Q_{ij}=Q_{ij}\cup C(\bar Q_{ij}^{(1)})$ and $\hat R_{ij}=R_{ij}\cup C(\bar R_{ij}^{(1)})$, where
$CP$ denotes the cone over the polyhedron $P$.
Then the map $(Q_{ij},R_{ij},q_{ij})\to (Q_{i,\,j+1},R_{i,\,j+1},q_{ij})$ factors through 
$(\hat Q_{ij},\hat R_{ij},q_{ij})$.
Also each bonding map $(Q_{i+1,j},R_{i+1,j})\to (Q_{ij},R_{ij})$ extends to a map 
$(\hat Q_{i+1,j},\hat R_{i+1,j})\to(\hat Q_{ij},\hat R_{ij})$.

Since $\pi_1(\hat Q_{ij},q_{ij})=\pi_1(\hat R_{ij},q_{ij})=1$, we get that the map
$\pi_2(\hat Q_{ij},q_{ij})\to\pi_2(\hat Q_{ij},\hat R_{ij},q_{ij})$ is surjective; 
the map $\pi_1(\hat Q_{ij},\hat R_{ij},q_{ij})\to\pi_0(\hat R_{ij},q_{ij})$ is injective; and
the abelian groups $\pi_n(\hat Q_{ij},q_{ij})$ and $\pi_n(\hat R_{ij},q_{ij})$, where $n\ge 2$, are finitely generated 
by Serre's theorem (see \cite{Sp}*{9.6.16}).
It follows that for $n\ge 2$ the groups $\pi_n(\hat Q_{ij},\hat R_{ij},q_{ij})$ are all abelian and finitely generated;
whereas the map $\lim_i\pi_1(\hat Q_{ij},\hat R_{ij},q_{ij})\to\lim_i\pi_0(\hat R_{ij},q_{ij})$ is injective.
But clearly $\pi_0(\hat R_{ij},q_{ij})\simeq\pi_0(R_{ij},q_{ij})$, and since $Y_j$ is locally connected, 
by Lemma \ref{continua} $\lim_i\pi_0(R_{ij},q_{ij})$ is finite.
Thus we get that $\lim_i\pi_1(\hat Q_{ij},\hat R_{ij},q_{ij})$ is also finite.
The remainder of the proof for $n\ge 1$ is now similar to the proof of Theorem \ref{ind-colimit}(b).
\end{proof}

\subsection{Pro-isomorphism}

Theorem \ref{ind-isomorphism2} is a consequence of the following result along with Lemma \ref{Cech}.

\begin{theorem} \label{ind-isomorphism1}
Let $f\:X\to Y$ be a map between local compacta, let 
$X_0\subset X_1\subset\dots$ and $Y_0\subset Y_1\subset\dots$ be compact subsets of 
$X$ and $Y$ respectively such that $\bigcup_i X_i=X$ and $\bigcup_i Y_i=Y$ and each 
$X_i\subset\Int X_{i+1}$ and each $Y_i\subset\Int Y_{i+1}$ and also each $f(X_i)\subset Y_i$.
Let $f_i\:X_i\to Y_i$ be the restriction of $f$.

(a) $f_*\:H_n(X)\to H_n(Y)$ is an isomorphism for all $n$ if and only if the induced maps $f_{i*}\:H_n(X_i)\to H_n(Y_i)$ 
represent an ind-isomorphism for all $n$.

(b) Suppose that $X$ and $Y$ are locally connected and $\breve\pi_1(X,x)=\breve\pi_1\big(Y,f(x)\big)=1$ for each $x\in X$.
Then $f_*\:\pi_n(X,x)\to\pi_n\big(Y,f(x)\big)$ is an isomorphism for all $n$ and all $x\in X$ if and only if 
for each $x\in X$ the induced maps $f_{i*}\:\pi_n(X_i,x)\to\pi_n\big(Y_i,f(x)\big)$, where $i\ge\min\{j\mid x\in X_j\}$, 
represent an ind-isomorphism for all $n$.
\end{theorem}

\begin{proof}[Proof. (a)] If $f_{i*}\:H_n(X_i)\to H_n(Y_i)$ represent an ind-isomorphism, then they induce an isomorphism
on direct limits, that is, $f_*\:H_n(X)\to H_n(Y)$ is an isomorphism.

Conversely, suppose that $f_*\:H_n(X)\to H_n(Y)$ is an isomorphism for each $n$.
Let $M$ be the metric mapping cylinder $MC(f)$ (see \cite{M00}*{\S\ref{book:mmc}}).
The composition of the inclusion $X\subset M$ and the deformation retraction $M\to Y$
equals $f$, so the inclusion induced map $H_n(X)\to H_n(M)$ is an isomorphism for each $n$.
Then $H_n(M,X)=0$ for all $n$.
Let $M_i=MC(f_i)$.
Then by Theorem \ref{ind-colimit2} for each $n$ and $j$ there exists a $k\ge j$ such that the inclusion induced map 
$H_n(M_j,X_j)\to H_n(M_k,X_k)$ is trivial.
Hence by Lemma \ref{ind-isomorphism} the induced maps $f_{i*}\:H_n(X_i)\to H_n(Y_i)$ 
represent an ind-isomorphism for all $n$.
\end{proof}

\begin{proof}[(b)] This is similar to the proof of (a), using additionally the following remarks.
Firstly, since $X$ and $Y$ are locally connected, so is $M$ (see \cite{M00}*{\S\ref{book:mmc}}).
Secondly, since $X$ is locally connected, each inclusion map $X_j\to X_{j+1}$ factors through 
a locally connected compactum $Z_j$ \cite{M1}*{proof of Theorem 6.11(a)}.
Then for each $x\in X_j$, every $\alpha\in\pi_0(Z_j,x)$ is represented by a point $z\in Z_j$.
Hence the image of $\alpha$ in $\pi_0(X_{j+1},x)$ is represented by the image of $z$.
\end{proof}

\begin{remark} Using Remark \ref{coherent} it is straightforward to extend Theorem \ref{ind-isomorphism1}
to the setting where $f$ is replaced by a coherent map in the sense of sequential strong antishape.
\end{remark}

\begin{remark} Using Proposition \ref{ind-0} it is straightforward to prove a version of 
Theorem \ref{ind-isomorphism1} where $f$ is replaced by a simplicial map between countable 
simplicial complexes.
Using this result, Sakai's approximation theorem (see \cite{M00}*{Corollary \ref{book:simp-appr}})
as well as \cite{M00}*{Theorem \ref{book:poly-cofinal}} and \cite{M00}*{Theorem \ref{book:mc-theorem}},
one can also deduce a version of Theorem \ref{ind-isomorphism1} where $f$ is replaced by 
a continuous map between separable polyhedra, and the direct sequences $X_0\subset X_1\subset\dots$
and $Y_0\subset Y_1\subset\dots$ are replaced by the direct systems of all compact subsets.
\end{remark}

\section{Whitehead-type theorem}

Theorem \ref{Wh-thm} is a consequence of the following result along with Lemma \ref{Cech}.

\begin{theorem} \label{Wh-thm1}
Let $X$ and $Y$ are locally connected finite-dimensional local compacta with trivial $\breve\pi_0$ and $\breve\pi_1$
and let $d=\max(\dim X+2,\,\dim Y+1)$.
If $f\:(X,x)\to (Y,y)$ is a fine shape morphism inducing a bijection on $\pi_n$ for $n<d$ and a surjection on $\pi_d$,
then $f$ is a fine shape equivalence.
\end{theorem}

When $f$ is represented by a map $(X,x)\to (Y,y)$, the proof can be slightly simplified: in this case there is an obvious shortcut
to the second half of Step 1, and in Step 2 one can refer directly to Theorem \ref{ind-colimit2} instead of repeating the arguments 
from its proof --- although much of these arguments will anyway have to be redone in order to achieve the desired stronger conclusions.

\begin{proof} The proof naturally splits into three steps.

\subsection{Step 1: Representing $f$ by a two-parameter h-level map}

Since $X$ is a local compactum, it can be represented as the union of a chain of its compact subsets $X_0\subset X_1\subset\dots$
such that each $X_j\subset\Int X_{j+1}$ (see \cite{M00}*{Proposition \ref{book:local compactum}}).
Then $X$ is homotopy equivalent to the mapping telescope $X_{[0,\infty)}$ (see \cite{M00}*{Proposition \ref{book:lc-telescope-a}}).
On the other hand, since $X$ is locally connected, each inclusion map $X_j\to X_{j+1}$ factors through a locally connected compactum 
$K_j$ \cite{M1}*{proof of Theorem 6.11(a)}.
It is not hard to see that $X_{[0,\infty)}$ is homotopy equivalent to $K_{[0,\infty)}$. 
(In fact, they are also homotopy equivalent to the mapping telescope of the direct sequence 
$X_0\subset K_0\to X_1\subset K_1\to\dots$.)
We may assume that $x\in X_0$, and then we also have $x\in K_0\subset K_1\subset\dots$.
Writing $x_i$ for the copy of $x$ in $K_i$, we get a ray $x_{[0,\infty)}\subset K_{[0,\infty)}$.
The previous arguments work to show that $(X,x)$ is homotopy equivalent to $(K_{[0,\infty)},\,x_{[0,\infty)})$,
as well as to $(K_{[0,\infty)},\,x_0)$.

Similarly, $Y$ is the union of a chain of its compact subsets $Y_0\subset Y_1\subset\dots$ such that $y\in Y_0$ and
each $Y_j\subset\Int Y_{j+1}$.
Each inclusion map $Y_j\to Y_{j+1}$ factors through a locally connected compactum $L_j$, and $Y$ is homotopy equivalent to 
the mapping telescope $L_{[0,\infty)}$.
Moreover, writing $y_i$ for the copy of $y$ in $L_i$, we have $(Y,Y)\simeq(L_{[0,\infty)},\,y_{[0,\infty)})$
and $(Y,y)\simeq(L_{[0,\infty)},\,y_0)$.

Let us represent $K_j$ as the limit of an inverse sequence of compact polyhedra $\PP_j=\big(\dots\to P_{1j}\to P_{0j}=pt\big)$.
Let $p_{ij}$ be the image of $x_j$ in $P_{ij}$.
Each composition $K_j\to X_{j+1}\to K_{j+1}$ extends to a map $P_{[0,\infty],\,j}\to P_{[0,\infty],\,j+1}$ of the extended mapping
telescopes which comes from an h-inv-map $\phi_j\:\PP_j\to\PP_{j+1}$ and sends the extended mapping telescope $p_{[0,\infty],\,j}$ 
of the singletons $\dots\to\{p_{1j}\}\to\{p_{0j}\}$ into $p_{[0,\infty],\,j+1}$ (see \cite{M00}*{Theorem \ref{book:strongshape}(a)
and Remark \ref{book:telescope0}} or \cite{M1}*{proof of Lemma 2.1(a)}).
By using Lemma \ref{straightening} we may assume that each $\phi_j$ is an h-level map.
The mapping telescope $P_{[0,\infty],\,[0,\infty)}$ of the direct sequence $\PP=\big(P_{[0,\infty],\,0}\to P_{[0,\infty],\,1}\to\dots\big)$ 
is an absolute retract which contains $K_{[0,\infty)}$ as a closed Z-set.
The polyhedron $P_{[0,\infty),\,[0,\infty)}$ will be also called $P$, and every its subpolyhedron of the form
$P_{[0,\infty),J}$, where $J\subset [0,\infty)$, will be also called $P_J$.

Similarly, we represent $L_j$ as the limit of an inverse sequence of compact polyhedra $\QQ_j=\big(\dots\to Q_{1j}\to Q_{0j}=pt\big)$
and let $q_{ij}$ be the image of $y_j$ in $Q_{ij}$.
Each composition $L_j\to Y_{j+1}\to L_{j+1}$ extends to a map $Q_{[0,\infty],\,j}\to Q_{[0,\infty],\,j+1}$ which comes from 
an h-level map $\psi_j\:\QQ_j\to\QQ_{j+1}$ and sends the base ray $q_{[0,\infty),\,j}$ into the base ray $q_{[0,\infty),\,j+1}$.
The mapping telescope $Q_{[0,\infty],\,[0,\infty)}$ of the direct sequence $\big(Q_{[0,\infty],\,0}\to Q_{[0,\infty],\,1}\to\dots\big)$ 
is an absolute retract which contains $L_{[0,\infty)}$ as a closed Z-set.
We will use the abbreviations $Q\bydef Q_{[0,\infty),\,[0,\infty)}$ and $Q_J\bydef Q_{[0,\infty),J}$.

The fine shape morphism $(K_{[0,\infty)},\,x_{[0,\infty)})\xr{\simeq}(X,x)\xr{f} (Y,y)\xr{\simeq}(L_{[0,\infty)},\,y_{[0,\infty)})$ 
is represented by a $K_{[0,\infty)}$-$L_{[0,\infty)}$-approaching map $F\:P\to Q$ which sends the base quadrant $p_{[0,\infty),\,[0,\infty)}$ 
into the base quadrant $q_{[0,\infty],\,[0,\infty)}$ (see \cite{M-I}*{Proposition \ref{fish:representation}}).
Since $F$ is $K_{[0,\infty)}$-$L_{[0,\infty)}$-approaching, it sends each $P_{[0,j]}$ into some $Q_{[0,k]}$, and by omitting some of 
the $Y_j$ we may assume that $k=j$.
Then upon amending $F$ by a $K_{[0,\infty)}$-$L_{[0,\infty)}$-approaching homotopy we may assume that
it additionally sends each $P_j$ into $Q_j$.
A further amendment of the same type ensures that $F$ sends each $P_{[j-1,\,j]}$ into $Q_{[j-1,\,j]}$.
Next, since $F$ is $K_{[0,\infty)}$-$L_{[0,\infty)}$-approaching, it sends each $P_{[j-1,\,j]}$ into $Q_{[j-1,\,j]}$ 
by a proper map, and hence for each $i$ there exists an $n_{ij}$ such that $F$ sends
$P_{[n_{ij},\,\infty),\,[j-1,\,j]}$ into $Q_{[i,\infty),\,[j-1,\,j]}$.
Since we are free to increase this $n_{ij}$, we may assume that each $n_{i+1,\,j}>n_{ij}$ and each $n_{ij}\ge n_{i,\,j-1}$.

Let $(Q'_{kj},q'_{kj})=(Q_{ij},q_{ij})$ for all $k\in (n_{ij},n_{i+1,\,j}]$.
The new bonding map $Q'_{kj}\to Q'_{k-1,\,j}$ is the original bonding map $Q_{ij}\to Q_{i-1,\,j}$ when $k$ is of the form $n_{ij}+1$,
and the identity map otherwise.
Thus we get a new inverse sequence $\QQ'_j=\big(\dots\to Q'_{1j}\to Q'_{0j}=pt\big)$.
If $k\in (n_{l,\,j-1},n_{l+1,\,j-1}]$ and $k\in (n_{ij},n_{i+1,\,j}]$, then $i\le l$ (indeed, $i\ge l+1$ would imply 
$k>n_{ij}\ge n_{l+1,\,j}\ge n_{l+1,\,j-1}\ge k$, which is contradictory).
In this case we can define $\psi'_{k,\,j-1}\:Q'_{k,\,j-1}\to Q'_{kj}$ to be the composition 
$Q_{l,\,j-1}\to Q_{i,\,j-1}\xr{\psi_{i,\,j-1}} Q_{ij}$.
These clearly combine into an h-level map $\psi'_{j-1}\:\QQ'_{j-1}\to\QQ'_j$.
This results in a map $Q'_{[0,\infty),\,j-1}\to Q'_{[0,\infty),\,j}$, which combines with the composition $K_{j-1}\to X_j\to K_j$
into a continuous map $Q'_{[0,\infty],\,j-1}\to Q'_{[0,\infty],\,j}$.
The mapping telescope $Q'_{[0,\infty],\,[0,\infty)}$ of the direct sequence $\big(Q'_{[0,\infty],\,0}\to Q'_{[0,\infty],\,1}\to\dots\big)$
is an absolute retract which contains $L_{[0,\infty)}$ as a closed Z-set.
We will use the abbreviations $Q'\bydef Q'_{[0,\infty),\,[0,\infty)}$ and $Q'_J\bydef Q'_{[0,\infty),J}$.

There is an obvious $L_{[0,\infty)}$-$L_{[0,\infty)}$-approaching homotopy equivalence 
$H\:Q\to Q'$ which sends every $Q_{[j-1,\,j]}$ into $Q'_{[j-1,\,j]}$ and the base quadrant $q_{[0,\infty),\,[0,\infty)}$ 
into the base quadrant $q'_{[0,\infty),\,[0,\infty)}$ (see Figure \ref{straightening-fig}).
Moreover, since $F$ sends each $P_{[n_{ij},\,\infty),\,[j-1,\,j]}$ into $Q_{(i-1,\,\infty),\,[j-1,\,j]}$, it is not hard 
to construct a $K_{[0,\infty)}$-$L_{[0,\infty)}$-approaching homotopy from $F$ to an $\tilde F$ such that
the composition $F'\:P\xr{\tilde F} Q\xr{H} Q'$ sends each $P_{[k-1,\infty),\,[j-1,j]}$ into 
$Q_{[k-1,\infty),\,[j-1,\,j]}$ (see Figure \ref{straightening-fig}).
Then upon amending $F'$ by a $K_{[0,\infty)}$-$L_{[0,\infty)}$-approaching homotopy we may assume that
it additionally sends each $P_{k-1,\,[j-1,\,j]}$ into $Q'_{k-1,\,[j-1,\,j]}$.
A further amendment of the same type ensures that $F'$ sends each $P_{[k-1,\,k],\,[j-1,\,j]}$ into 
$Q'_{[k-1,\,k],\,[j-1,\,j]}$.
This implies that it also sends each $P_{k,\,[j-1,\,j]}$ into $Q'_{k,\,[j-1,\,j]}$, each
$P_{[k-1,\,k],\,j}$ into $Q'_{[k-1,\,k],\,j}$ and each $P_{kj}$ into $Q'_{kj}$.

\begin{figure}[h]
\includegraphics[width=14.5cm]{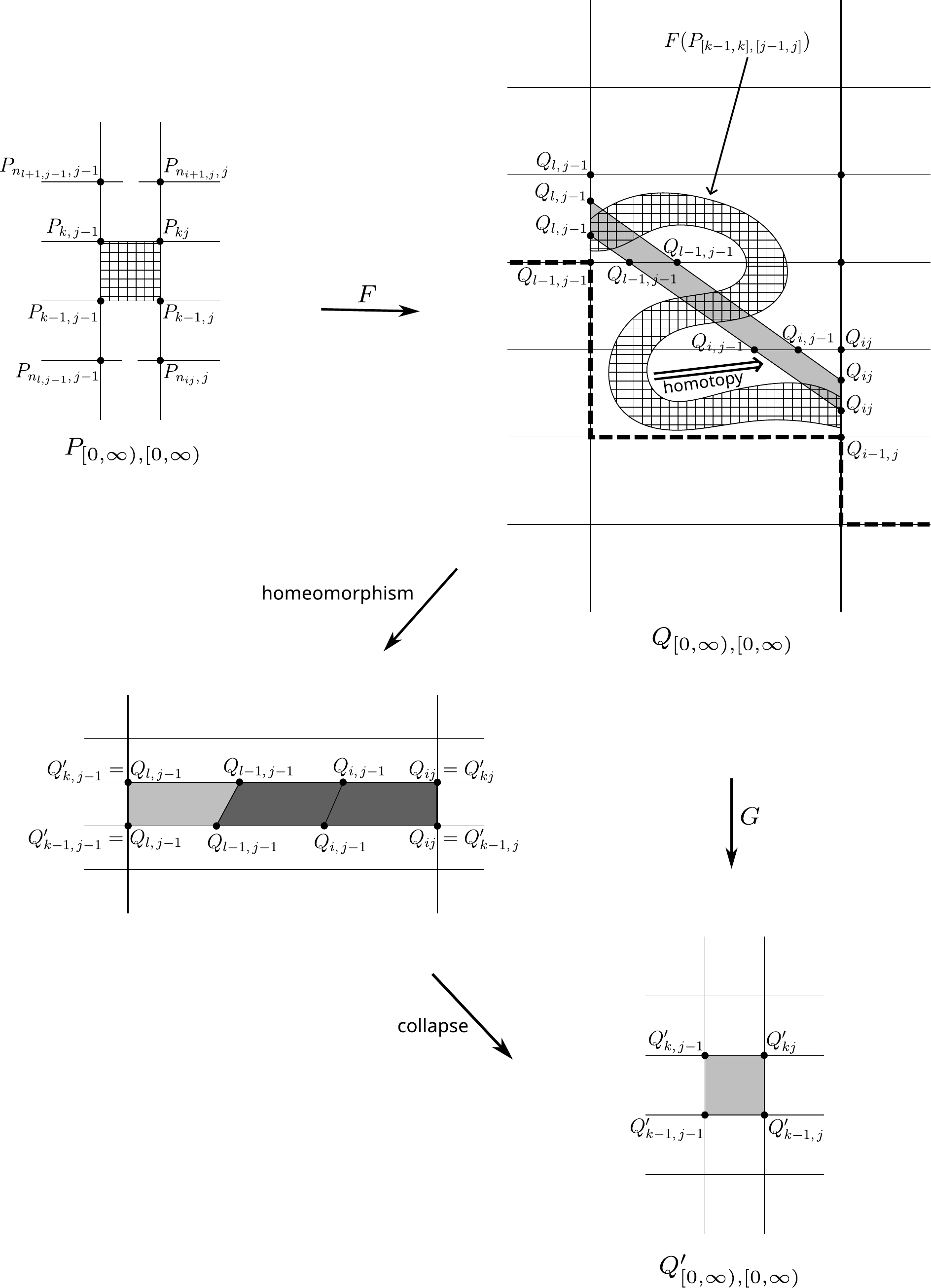}
\caption{The maps $F$ and $G$ and the homotopy from $F$ to $\tilde F$.}
\label{straightening-fig}
\end{figure}

In what follows we will no longer need $Q_{ij}$, $q_{ij}$, $\psi_{ij}$ and $F$. 
So we will recycle this notation to denote $Q'_{ij}$, $q'_{ij}$, $\psi'_{ij}$ and $F'$, respectively, 
and similarly for their derivatives.

As noted above, $F$ restricts to maps $F_{ij}\:P_{ij}\to Q_{ij}$ and yields homotopies 
$h_{ij}\:P_{ij}\x I\to Q_{i+1,\,j}$ and $h'_{ij}\:P_{ij}\to Q_{i,\,j+1}$
which make the corresponding square diagrams homotopy commutative, and also to $2$-homotopies 
$H_{ij}\:P_{ij}\x I^2\to Q_{i+1,\,j+1}$ which extend the six homotopies on the faces of the cube
(namely, $h_{ij}$, $h_{i,\,j+1}$, $h'_{ij}$, $h'_{i+1,\,j}$ and the homotopies associated with 
the h-level maps $\phi_j$ and $\psi_j$).
Upon amending $F$ by a yet another $K_{[0,\infty)}$-$L_{[0,\infty)}$-approaching homotopy 
we may assume that it is the map which comes from the $F_{ij}$, $h_{ij}$, $h'_{ij}$ and $H_{ij}$.
Since $F$ sends the base quadrant $p_{[0,\infty),\,[0,\infty)}$ into the base quadrant 
$q_{[0,\infty),\,[0,\infty)}$, each $F_{ij}$ sends the basepoint $p_{ij}$ into the basepoint $q_{ij}$,
and the homotopies $h_{ij}$ and $h'_{ij}$ and the $2$-homotopy $H_{ij}$ similarly respect the basepoints.

\subsection{Step 2: Algebra}

For a subset $J\subset [0,\infty)$ let $F_J\:P_J\to Q_J$ be the restriction of $F$ and let $M_J$ 
be the metric mapping cylinder $MC(F_J)$ (see \cite{M00}*{\S\ref{book:mmc}}).
Also, for a $j\in [0,\infty)$ let $p_j=p_{[0,\infty),j}$ and $q_j=q_{[0,\infty),j}$,
and let $m_j$ be the metric mapping cylinder of $F|_{p_j}\:{p_j}\to {q_j}$.

The composition $(P,p_0)\subset(M,p_0)\subset(M,m_0)\xr{r} (Q,q_0)$,
where $r$ is a deformation retraction and the second inclusion is also a homotopy equivalence,
coincides with $F\:(P,p_0)\to (Q,q_0)$, which in turn represents the fine shape morphism 
$(K_{[0,\infty)},\,x_0)\xr{\simeq}(X,x)\xr{f} (Y,y)\xr{\simeq}(L_{[0,\infty)},\,y_0)$.
Hence the inclusion induced map $\Pi_n(P,p_0)\to\Pi_n(M,p_0)$ is a bijection
for each $n<d$ and a surjection for $n=d$, where $d=\max(\dim X+2,\,\dim Y+1)=\dim M-1$.
Then $\Pi_n(M,P,p_0)$ is trivial for $n\le d$.
On the other hand, clearly $\Pi_n(M,P,p_0)\simeq\colim\Pi_n(M_{[0,j]},P_{[0,j]},p_0)$.
Each $P_{[0,j]}$ deformation retracts onto $P_j$, each $Q_{[0,j]}$ deformation retracts onto $Q_j$,
and each $(M_{[0,j]},M_j)$ deformation retracts onto $(Q_{[0,j]},Q_j)$.
Using these deformation retractions, it is not hard to see that each $(M_{[0,j]},P_{[0,j]},p_0)$ 
is homotopy equivalent to $(M_j,P_j,p_j)$, and moreover these homotopy equivalences commute up to
homotopy with the bonding maps.
Hence $\colim\Pi_n(M_j,P_j,p_j)$ is trivial for each $n\le d$.

Let $M_{ij}=MC(F_{ij})$.
Lemma \ref{2-mc} yields an inverse sequence $\dots\to M_{1j}\to M_{0j}$ and identifies $M_j$ 
with its mapping telescope.
Then we have the following commutative diagram with exact rows:
\[\begin{tikzcd}[column sep=2em]
1\rar &\lim\limits_i\!^1\,\pi_{n+1}(M_{ij},\,P_{ij},\,p_{ij})\rar\dar &\Pi_n(M_j,P_j,p_j)
\rar\dar &\lim\limits_i\pi_n(M_{ij},\,P_{ij},\,p_{ij})\rar\dar &1\\
1\rar &\lim\limits_i\!^1\,\pi_{n+1}(M_{ik},\,P_{ik},\,p_{ik})\rar &\Pi_n(M_k,P_k,p_j)
\rar &\lim\limits_i\pi_n(M_{ik},\,P_{ik},\,p_{ik})\rar &1
\end{tikzcd}\] 
where $k\ge j$ and $n\ge 1$ (see \cite{M1}*{proof of Theorem 3.1(b)}).
Writing $G_{ij}^n=\pi_n(M_{ij},P_{ij},p_{ij})$, from this diagram we conclude that 
$\colim_j\lim_i G^n_{ij}=1$ for $1\le n\le d$ and that $\colim_j\lim^1_i G^n_{ij}=0$ for $2\le n\le d+1$.

We have $\colim\check\pi_0(K_j,x_j)\simeq\colim\check\pi_0(X_j,x_j)\simeq\breve\pi_0(X,x)=*$.
Since $K_j$ is locally connected, by Lemma \ref{continua} $\check\pi_0(K_j,x_j)$ is finite.
Then for each $j$ there exists a $k$ such that the map $\check\pi_0(K_j,x_j)\to\check\pi_0(K_k,x_j)$ 
is trivial (see Proposition \ref{obvious}).
By omitting some of the $X_i$ we may assume that $k=j+1$.
Writing $D_{ij}=\pi_0(P_{ij},p_{ij})$, we get that the map $\lim_i D_{ij}\to\lim_i D_{i,\,j+1}$ is trivial.
If $E_{ij}$ denotes the image of the map $D_{ij}\to D_{i,\,k_j}$, then by Lemma \ref{level-factorization4} 
$\lim_i E_{ij}$ is trivial.
The tower of finite sets $\dots\to E_{i1}\to E_{i0}$ also satisfies the Mittag-Leffler condition, 
and hence it is trivial as a pro-pointed-set, 
that is, for each $i$ there exists an $l>i$ such that the map $E_{lj}\to E_{ij}$ is trivial 
(see \cite{M1}*{Lemma 3.4(a)}).
Then the composition $D_{lj}\to E_{lj}\to E_{ij}\to D_{i,\,j+1}$ is trivial.
Using Lemma \ref{straightening} (applied simultaneously to the $P_{ij}$ and $Q_{ij}$, so as to keep $F$), 
we may assume that already the maps $D_{ij}\to D_{i,\,j+1}$ are trivial.
Similarly, since each $L_j$ is locally connected, we may assume that the maps 
$\pi_0(Q_{ij},q_{ij})\to\pi_0(Q_{i,\,j+1},q_{i,\,j+1})$ are trivial.
Then the maps $\pi_0(M_{ij},p_{ij})\to\pi_0(M_{i,\,j+1},p_{i,\,j+1})$ are also trivial.

Since the $K_j$ and $L_j$ are locally connected, and $\breve\pi_1(X,x)$ and $\breve\pi_1(Y,y)$ are trivial,
similarly to the proof of Theorem \ref{ind-colimit}(b) we may assume that the maps 
$\pi_1(P_{ij},p_{ij})\to\pi_1(P_{i,\,j+1},p_{i,\,j+1})$ and
$\pi_1(Q_{ij},q_{ij})\to\pi_1(Q_{i,\,j+1},q_{i,\,j+1})$ are trivial.
Then so are the maps $\pi_1(M_{ij},p_{ij})\to\pi_1(M_{i,\,j+1},p_{i,\,j+1})$.

Let $(T,S)$ be a triangulation of $(M_{ij},P_{ij})$ which has $p_{ij}$ as a vertex, and
let $M_{ij}^{(k)}$ and $P_{ij}^{(k)}$ denote the $k$-skeleta of $T$ and $S$.
Then the map $(M_{ij},P_{ij},p_{ij})\to (M_{i,\,j+1},P_{i,\,j+1},p_{i,\,j+1})$ factors through 
$(M_{ij}\cup C(M_{ij}^{(0)}),P_{ij}\cup C(P_{ij}^{(0)}),\,p_{ij})$, where
$CQ$ denotes the cone over the polyhedron $Q$.
Consequently the map $(M_{ij},P_{ij},p_{ij})\to (M_{i,\,j+2},P_{i,\,j+2},p_{i,\,j+2})$ factors through 
$(\hat M_{ij},\hat P_{ij},\,p_{ij})$, where $\hat P_{ij}=P_{ij}\cup C(P_{ij}^{(1)})$ and 
$\hat M_{ij}=M_{ij}\cup C(M_{ij}^{(1)})$.
Also each bonding map $(M_{i+1,j},P_{i+1,j})\to (M_{ij},P_{ij})$ extends to a map 
$(\hat M_{i+1,j},\hat P_{i+1,j})\to(\hat M_{ij},\hat P_{ij})$.

Let $\hat G_{ij}^n=\pi_n(\hat M_{ij},\hat P_{ij},p_{ij})$.
Since $\pi_1(\hat M_{ij},p_{ij})=1$, the boundary map $\pi_2(\hat P_{ij},p_{ij})\to\pi_2(\hat M_{ij},\hat P_{ij},p_{ij})$ is surjective, 
and hence the group $\hat G_{ij}^2$ is abelian.
Since $\pi_1(\hat M_{ij},p_{ij})=1$ and $\pi_0(\hat M_{ij},p_{ij})=*$, we have $\hat G_{ij}^1=1$.
Finally, since $\pi_0(\hat P_{ij},p_{ij})=*$, we have $\hat G_{ij}^0=*$.

Since $\pi_1(\hat M_{ij},p_{ij})=1$ and $\pi_1(\hat P_{ij},p_{ij})=1$, the abelian groups $\pi_n(\hat M_{ij},p_{ij})$ and 
$\pi_n(\hat P_{ij},p_{ij})$, where $n\ge 2$, are finitely generated by Serre's theorem (see \cite{Sp}*{9.6.16}).
Therefore for $n\ge 2$ the groups $\hat G_{ij}^n$ are all abelian and finitely generated.
Since $\colim_j\lim_i\hat G^n_{ij}\simeq\colim_j\lim_i G^n_{ij}=1$, by Theorem \ref{bounded-colim}(a) for each $j$ 
there exists a $k>j$ such that the map $\lim_i\hat G^n_{ij}\to\lim_i\hat G^n_{ik}$ is trivial for $2\le n\le d$.
By omitting some of the $X_k$ and $Y_k$ we may assume that $k=j+1$.
Since $\colim_j\lim^1_i\hat G^n_{ij}\simeq\colim_j\lim^1_i G^n_{ij}=1$, by Theorem \ref{bounded-colim}(c) for each $j$ 
there exists a $k>j$ such that the map $\lim^1_i\hat G^n_{ij}\to\lim^1_i\hat G^n_{ik}$ is trivial for $2\le n\le d$.
By omitting some of the $X_k$ and $Y_k$ we may assume that $k=j+1$.

Thus the level map $\hat G_{ij}^n\to\hat G_{i,\,j+1}^n$ induces trivial maps on $\lim$ and on $\lim^1$ for $2\le n\le d$.
Then by Lemma \ref{level-factorization}(c) it factors through a tower of abelian groups $H^n_{ij}$ such that both 
$\lim_i H^n_{ij}$ and $\lim^1_i H^n_{ij}$ are trivial. 
Then the tower $\dots\to H^n_{1j}\to H^n_{0j}$ satisfies the Mittag-Leffler condition (see \cite{M00}*{Theorem \ref{book:gray}(b)}),
and hence is trivial as a pro-group, that is, for each $i$ there exists an $l>i$ such that the map $H^n_{lj}\to H^n_{ij}$ 
is trivial (see \cite{M1}*{Lemma 3.4(a)}) for $2\le n\le d$.
Therefore the composition $\hat G^n_{lj}\to H^n_{lj}\to H^n_{ij}\to\hat G^n_{i,\,j+1}$ is trivial for $2\le n\le d$.
Using Lemma \ref{straightening} (applied simultaneously to the $P_{ij}$ and $Q_{ij}$) we may assume that already the maps 
$\hat G^n_{ij}\to\hat G^n_{i,\,j+1}$ are trivial for $2\le n\le d$.
The latter also holds for $n=0,1$ since $G^0_{ij}=*$ and $G^1_{ij}=1$.
Then also the compositions $G^n_{ij}\to\hat G^n_{ij}\to\hat G^n_{i,\,j+1}\to G^n_{i,\,j+3}$ are trivial for $n\le d$.
By omitting some of the $X_i$ and $Y_i$ we may assume that already the maps $G^n_{ij}\to G^n_{i,\,j+1}$ are trivial for $n\le d$.

\subsection{Step 3: Construction of an inverse of $f$ in fine shape}

This step is similar to \cite{M1}*{proof of Theorem 3.6}.
By an induction on $k=0,1,\dots,d-1$, the bonding map $\mu_{k+1}\:M_{ij}\to M_{i,\,j+k+1}$ 
is homotopic keeping $P_{ij}$ fixed to a map sending the $k$-skeleton $M_{ij}^{(k)}$ into $P_{i,\,j+k+1}$.
Since $\dim M_{ij}\le d-1$, we eventually obtain a homotopy $\rho_{ij}$ keeping $P_{ij}$ fixed from 
the bonding map $\mu_d\:M_{ij}\to M_{i,\,j+d}$ to a map $M_{ij}\to P_{i,\,j+d}$.

Given subcomplexes $I$, $J$ of the triangulation of $[0,\infty)$ with integer vertices, let $M_{I,J}$ denote 
the mapping cylinder of the restriction $P_{I,J}\to Q_{I,J}$ of $F$.
Since $\dim M_{[i,\,i+1],\,j}\le d$, by an inductive construction just like the previous one we extend the homotopies 
$\mu_{d+1} \rho_{ij}$ and $\mu_{d+1} \rho_{i+1,\,j}$ to a homotopy $\rho_{[i,\,i+1],\,j}$ keeping $P_{[i,\,i+1],\,j}$ fixed
from the bonding map $\mu_{2d+1}\:M_{[i,\,i+1],\,j}\to M_{[i,\,i+1],\,j+2d+1}$ to a map 
$M_{[i,\,i+1],\,j}\to P_{[i,\,i+1],\,j+2d+1}$.
In a similar way we also extend the homotopies $\mu_{d+1} \rho_{ij}$ and $\mu_{d+1} \rho_{i,\,j+1}$
to a homotopy $\rho_{i,\,[j,\,j+1]}$ keeping $P_{i,\,[j,\,j+1]}$ fixed from the bonding map 
$\mu_{2d+1}\:M_{i,\,[j,\,j+1]}\to M_{i,\,[j+2d+1,\,j+2d+2]}$ to a map $M_{i,\,[j,\,j+1]}\to P_{i,\,[j+2d+1,\,j+2d+2]}$.

Finally, let $M_{\square ij}=M_{\{i,\,i+1\},\,[j,\,j+1]}\cup M_{[i,\,i+1],\,\{j,\,j+1\}}$ 
and let $\rho_{\square ij}\:M_{\square ij}\to M_{\square i,\,j+2d+1}$ denote the homotopy $\rho_{[i,\,i+1],\,j}\cup \rho_{i,\,[j,\,j+1]}\cup
\rho_{[i,\,i+1],\,j+1}\cup \rho_{i+1,\,[j,\,j+1]}$.
Then by a similar construction we extend the composition
$M_{\square ij}\xr{\rho_{\square ij}} M_{\square i,\,j+2d+1}\xr{\mu_{d+1}} M_{\square i,\,j+3d+2}$
to a homotopy $\rho_{[i,\,i+1],\,[j,\,j+1]}$ keeping $P_{[i,\,i+1],\,[j,\,j+1]}$ fixed from the bonding map 
$\mu_{3d+2}\:M_{[i,\,i+1],\,[j,\,j+1]}\to M_{[i,\,i+1],\,[j+3d+2,\,j+3d+3]}$ to a map sending
$M_{[i,\,i+1],\,[j,\,j+1]}^{(d)}$ into $P_{[i,\,i+1],\,[j+3d+2,\,j+3d+3]}$.

The bonding maps of the form $\mu_{3d+2}\:M_{[i,\,i+1],\,[j,\,j+1]}\to M_{[i,\,i+1],\,[j+3d,\,j+3d+1]}$ glue together 
into a self-map $\mu_{3d+2}$ of $M$, also called $\sigma$, which is homotopic to $\id_M$ by a homotopy $\Sigma$.
Moreover, $\Sigma$ moves $P$ within itself and moves $Q$ within itself.
The homotopies $\rho_{[i,\,i+1],\,[j,\,j+1]}$ glue together into a homotopy $\rho'$ keeping $P$ fixed from $\sigma\:M\to M$ 
to a map $r'$ sending $M^{(d)}$ into $P$.
Since $\dim M=d+1$, using the condition $\colim_j\lim^1_i G^{d+1}_{ij}=0$ (which we have not used so far) it is not hard 
to amend $r'$ by a homotopy keeping $P$ fixed and sending every $M_{[j,j+1]}$ into $M_{[j,k]}$ for some $k=k(j)$
so that the resulting map $r$ sends the entire $M$ into $P$. 
(The parallel step in \cite{M1}*{proof of Theorem 3.6} is described in more detail.)
Thus we obtain a homotopy $\rho$ keeping $P$ fixed from $\sigma$ to a map $r\:M\to P$.

The restriction $G\:Q\to P$ of $r$ is an $L_{[0,\infty)}$-$K_{[0,\infty)}$-approaching map.
If $\Phi\:P\x I\to M$ is the natural homotopy from the inclusion map to $F$, then the composition
$P\x I\xr{\Phi} M\xr{r}P$ is a $K_{[0,\infty)}$-$K_{[0,\infty)}$-approaching homotopy from $r|_P$ to $GF$.
Since $\rho$ keeps $P$ fixed, $r|_P=\sigma|_P$.
So $\Sigma$ restricts to a $K_{[0,\infty)}$-$K_{[0,\infty)}$-approaching homotopy from $r|_P$ to $\id_P$.
On the other hand, $\rho$ restricts to a homotopy from $G$ to $\sigma|_Q$.
If $R\:M\to Q$ is the natural retraction, then $R\rho$ restricts to an $L_{[0,\infty)}$-$L_{[0,\infty)}$-approaching 
homotopy from $FG$ to $\sigma|_Q$.
Also $\Sigma$ restricts to an $L_{[0,\infty)}$-$L_{[0,\infty)}$-approaching homotopy from $\sigma|_Q$ to $\id_Q$.
The constructed maps and homotopies respect the base quadrants $p_{[0,\infty),[0,\infty)}$ and $q_{[0,\infty),[0,\infty)}$,
and it follows that $f$ is a pointed fine shape equivalence.
\end{proof}

\subsection*{Disclaimer}

I oppose all wars, including those wars that are initiated by governments at the time when 
they directly or indirectly support my research. The latter type of wars include all wars 
waged by the Russian state in the last 25 years (in Chechnya, Georgia, Syria and Ukraine) 
as well as the USA-led invasions of Afghanistan and Iraq.


\begin{thebibliography}{00}

\bib{Bri}{article}{
   author={Bridson, M. R.},
   title={Doubles, finiteness properties of groups, and quadratic isoperimetric inequalities},
   journal={J. Algebra},
   volume={214},
   date={1999},
   pages={652--667},
}

\bib{Koy}{article}{
   author={Koyama, A.},
   title={A Whitehead-type theorem in fine shape theory},
   journal={Glas. Mat. Ser. III},
   volume={18},
   date={1983},
   number={2},
   pages={359--370},
}

\bib{Li}{article}{
   author={Lisitsa, Yu. T.},
   title={Hurewicz and Whitehead theorems in the strong shape theory},
   language={Russian},
   journal={Dokl. Akad. Nauk SSSR},
   volume={283},
   date={1985},
   number={1},
   pages={38--43},
   translation={
   journal={Soviet Math. Dokl.},
      volume={32},
      date={1985},
      pages={31--35},
   }
}

\bib{MKS}{book}{
   author={Magnus, W.},
   author={Karrass, A.},
   author={Solitar, D.},
   title={Combinatorial Group Theory},
   publisher={Interscience [John Wiley \& Sons]},
   place={New York},
   date={1966},
   translation={
    language={Russian},
    publisher={Nauka},
    place={Moscow},
    date={1974}
    }
}

\bib{M1}{article}{
   author    = {S. A. Melikhov},
   title     = {Steenrod homotopy},
    journal  = {Russ. Math. Surv.},
       volume   = {64},
       year     = {2009},
       pages    = {469--551},
   note    = {\arxiv{0812.1407}},
   translation = {
    language = {Russian},
      journal   = {Uspekhi Mat. Nauk},
      volume    = {64:3},
      year      = {2009},
      pages     = {73--166},
   }
}

\bib{M00}{article}{
   author    = {S. A. Melikhov},
   title     = {Topology of Metric Spaces},
   note={\href{https://www.researchgate.net/publication/365476532_Topology_of_Metric_Spaces}{ResearchGate}},
   date={18.11.2022}
}

\bib{M-I}{article}{
   author    = {S. A. Melikhov},
   title     = {Fine shape. I},
   note={\arxiv{1808.10228v2}},
}

\bib{M-V}{article}{
   author    = {S. A. Melikhov},
   title     = {Coronated polyhedra and coronated ANRs},
   note={\arxiv{?}},
}

\bib{Mi1}{report}{
   author        = {J. Milnor},
   title         = {On the Steenrod homology theory},
   address       = {Berkeley},
   year          = {1961},
   reprint       = {
    title         = {Novikov Conjectures, Index Theorems and Rigidity (vol. 2)},
    pages         = {79--96},
    publisher     = {Cambridge Univ. Press},
    year          = {1995},
    series        = {London Math. Soc. Lecture Note Ser.},
    editor        = {S. C. Ferry},
    editor        = {A. Ranicki},
    editor        = {J. Rosenberg},
    volume        = {226},
    note          = {\href{https://www.maths.ed.ac.uk/~v1ranick/books/novikov2.pdf}{Andrew Ranicki's archive}},
   },
}

\bib{Qu}{article}{
   author={Quigley, J. B.},
   title={An exact sequence from the $n$th to the $(n-1)$-st fundamental
   group},
   journal={Fund. Math.},
   volume={77},
   date={1973},
   number={3},
   pages={195--210},
}

\bib{Sp}{book}{
   author    = {E. Spanier},
   title     = {Algebraic topology},
   publisher = {McGraw--Hill},
   address   = {New York},
   year      = {1966},
}

\bib{Ste}{article}{
   author={Steenrod, N. E.},
   title={Regular cycles of compact metric spaces},
   journal={Ann. of Math.},
   volume={41},
   date={1940},
   pages={833--851},
}

\end{thebibliography}
\end{document}